\documentclass[12pt, a4paper]{amsart}

\usepackage[utf8]{inputenc}
\usepackage[english]{babel}
\usepackage{amsmath, amssymb, amsfonts, amsthm}
\usepackage[margin=2cm]{geometry}
\usepackage{hyperref}
\hypersetup{
    colorlinks=true,
    linkcolor= blue,
    citecolor=blue
}
\usepackage[all]{xy}
\usepackage{enumerate}

\theoremstyle{plain}
\newtheorem{theorem}{Theorem}[section]

\newtheorem{proposition}[theorem]{Proposition}
\newtheorem{lemma}[theorem]{Lemma}
\newtheorem{corollary}[theorem]{Corollary}

\theoremstyle{definition}
\newtheorem{example}[theorem]{Example}

\newtheorem{assumption}[theorem]{Assumption}

\theoremstyle{remark}
\newtheorem{remark}[theorem]{Remark}

\numberwithin{equation}{section}

\newcommand{\B}{{\mathcal{B}}}
\newcommand{\Sval}{{\rm S}_{\rm val}}
\newcommand{\Sord}{{\rm S}_{\rm ord}}
\newcommand{\Sall}{{\rm S}}
\newcommand{\Sabsval}{{\rm S}_{\rm abs}}
\newcommand{\BB}{\mathfrak{B}}

\newcommand{\MM}{\mathfrak{m}}
\newcommand{\Mcal}{\mathcal{M}}

\newcommand{\condI}{{\rm(I)}}
\newcommand{\condU}{{\rm(U)}}
\newcommand{\condT}{{\rm(T)}}

\title[Approximation theorems]{Approximation theorems for spaces of localities}
\author{Sylvy Anscombe, Philip Dittmann, and Arno Fehm}
\address{Jeremiah Horrocks Institute, University of Central Lancashire, Preston PR1 2HE, United Kingdom}
\email{sanscombe@uclan.ac.uk}
\address{Afdeling Algebra, KU Leuven, Celestijnenlaan 200b, 3001 Leuven, Belgium}
\email{philip.dittmann@kuleuven.be}
\address{Institut f\"{u}r Algebra, Technische Universit\"{a}t Dresden, 01062 Dresden, Germany}
\email{arno.fehm@tu-dresden.de}
\thanks{\today}

\begin{document}
\begin{abstract}
The classical Artin--Whaples approximation theorem allows to simultaneously approximate finitely many different elements of a field with respect to finitely many pairwise inequivalent absolute values.
Several variants and generalizations exist, for example for finitely many (Krull) valuations,
where one usually requires that these are independent, i.e.~induce different topologies on the field.
Ribenboim proved a generalization for finitely many valuations where the condition of independence is relaxed for a natural compatibility condition, and Ershov proved a statement about simultaneously approximating finitely many different elements with respect to finitely many possibly infinite sets of pairwise independent valuations.

We prove approximation theorems for infinite sets of valuations and orderings without requiring pairwise independence.
\end{abstract}
\maketitle

\setcounter{tocdepth}{2}

\section{Introduction}
We fix a field $K$,
elements $x_1,\dots,x_n\in K$ and $z_1,\dots,z_n\in K^\times$
and start by recalling the classical approximation theorem for absolute values: 

\begin{theorem}[Artin--Whaples 1945\footnote{Published in \cite{AW}. See also the historical remarks in \cite[\S4.2.1]{Roquette}. See also {\cite[XII.1.2]{Lang},\cite[1.1.3]{EP}}.}] \label{thm:artin_whaples}
Let $|.|_1,\dots,|.|_n$ be nontrivial absolute values on $K$.
\begin{enumerate}
\item[$\condI$] Assume that  $|.|_1,\dots,|.|_n$ are pairwise inequivalent.
\end{enumerate}
Then there exists $x\in K$ with
$$
 |x-x_i|_i < |z_i|_i\mbox{ for }i=1,\dots,n.
$$
\end{theorem}

Since the non-archimedean absolute values correspond to Krull valuations of rank $1$,
the following theorem is a generalization in the non-archimedean case:

\begin{theorem}[{Bourbaki\footnote{This attribution is taken from \cite{Ribenboim}.
See also 
\cite[VI.7.2 Theorem 1] {Bourbaki},
\cite[2.4.1]{EP}, \cite[10.1.7]{Efrat}.}}]
\label{thm:Bourbaki}
Let $v_1,\dots,v_n$ be nontrivial valuations\footnote{The term {\em valuation} in this work always refers to Krull valuations,
with value group written additively.} on $K$.
\begin{enumerate}
    \item[$\condI$] Assume that $v_1,\dots,v_n$ are pairwise independent\footnote{That is, they induce distinct topologies on $K$, or, equivalently, they have no nontrivial common coarsening.}.
\end{enumerate}
Then there exists $x\in K$ with
$$
    v_i(x-x_i) > v_i(z_i) \mbox{ for }i=1,\dots,n.
$$
\end{theorem}
In the literature, one can find three possible directions of generalizing Theorem \ref{thm:Bourbaki}:

Firstly, one can unify Theorems \ref{thm:Bourbaki} and \ref{thm:artin_whaples} to approximate with respect to finitely many pairwise independent valuations and absolute values,
like in \cite[(4.2) Corollary]{PZ} or \cite[Cor.~27.14]{Warner}.
This also includes approximation with respect to (not necessarily archimedean) orderings on $K$, see Theorem \ref{thm:pairwise_indep_fin_approx} below.

Secondly, one can relax the condition of pairwise independence in Theorem \ref{thm:Bourbaki}, at the expense of introducing a compatibility condition,
as done by Nagata \cite{Nagata} and Ribenboim \cite{Ribenboim}.
\begin{theorem}[Ribenboim 1957\footnote{This is a mild reformulation of {\cite[Th\'{e}or\`{e}me 5’]{Ribenboim}}, see also \cite[p.~136 Th\'{e}or\`{e}me 3]{Ribenboim_book}.}]
\label{thm:Ribenboim}
Let $v_1,\dots,v_n$ be valuations on $K$.
\begin{enumerate}
    \item[$\condI$] Assume that $v_1,\dots,v_n$ are pairwise incomparable
and that if $w$ is a common coarsening of $v_i$ and $v_j$, $i\neq j$, 
then $w(x_i-x_j)\geq w(z_i)=w(z_j)$.
\end{enumerate}
Then there exists $x\in K$ with
$$
    v_i(x-x_i) > v_i(z_i) \mbox{ for }i=1,\dots,n.
$$
\end{theorem}
Thirdly, and most pertinently for us, one can find approximation theorems for 
possibly infinite subsets of the spaces $\Sord(K)$ of orderings or
 $\Sval(K)$ of valuations on $K$,
under
some
assumptions on
(U)niformity, (T)opology and (I)ndependence (our presentation).
The first occurence of such a `block approximation' theorem seems to be
\cite[Theorem 1]{ErshovTotallyRealFieldExtns},
but there, like
in many subsequent works including
\cite{PrestelOnTheAxiomatisationOfPRCFields,
FHV, DarniereLGP, ErshovTehran, HJP},
these results apply only to very special fields
satisfying some geometric local-global principle, 
an example of which is the following:

\begin{theorem}[Prestel 1985\footnote{See \cite[p.~354]{PrestelOnTheAxiomatisationOfPRCFields}, also reproven in \cite[Corollary 1.3]{FHV}.}]\label{thm:intro_Prestel}
Let $S_1,\dots,S_n\subseteq\Sord(K)$ pairwise disjoint.
\begin{enumerate}
    \item[$\condT$] Assume that each $S_i$ is compact in the Harrison topology on $\Sord(K)$ (cf.~Section \ref{sec:topologies}).
    \item[$\condI$] Assume that $K$ is pseudo-real closed (PRC), that is, every geometrically integral $K$-variety that has rational points over every real closure of $K$ has a $K$-rational point.\footnote{The PRC property can indeed be seen as a very strong independence assumption, since it implies in particular that distinct orderings on $K$ induce distinct topologies, cf.~\cite[p.~353]{PrestelOnTheAxiomatisationOfPRCFields}.}
\end{enumerate}
Then there exists $x \in K$ with 
$$
 (x-x_i)^2 < z_i^2 \text{ for all orderings in } S_i, \mbox{ for } i=1,\dots,n.
$$ 
\end{theorem}

However, for valuations, Darni\`ere \cite{DarniereThese} and Ershov \cite{ErshovMultiValuedFields}  proved results for general fields. 

\begin{theorem}[Ershov 2001\footnote{See discussion in Section \ref{sec:Ershov} how this follows from the results in \cite{ErshovMultiValuedFields}.}]\label{thm:intro_Ershov_uniformizer}
  Let $S_1,\dots,S_n\subseteq\Sval(K)$ pairwise disjoint.
\begin{enumerate}
    \item[$\condU$]Assume that there exists a common uniformizer $\pi\in K$ of all $v\in S_1\cup\dots\cup S_n$.
    \item[$\condT$] Assume that each $S_i$ is compact\footnote{We use ``compact'' to mean what other sources call ``quasi-compact'', i.e.\ there is no implication of being a Hausdorff space.} in the Zariski topology on $\Sval(K)$ (cf.~Section \ref{sec:topologies}).
    \item[$\condI$] Assume that 
    the elements of $ S_1\cup\dots\cup S_n$ are pairwise independent.
\end{enumerate}
Then there exists $x\in K$ with
$$
v(x-x_i)\geq v(z_i)\mbox{ for all }v\in S_i,\mbox{  for }i=1,\dots,n.
$$
\end{theorem}

The aim of this paper is to prove approximation theorems that generalize Theorem \ref{thm:Bourbaki} in these directions simultaneously; they apply to infinite sets of valuations and orderings, without assuming pairwise independence
or the presence of a local-global principle. 
Our main result in this direction is
Theorem \ref{thm:main},
which is in particular a common generalization of
Theorem \ref{thm:intro_Prestel} and Theorem \ref{thm:intro_Ershov_uniformizer}.
Instead of attempting to explain the technical assumptions of that main theorem here,
we now simply quote three instances of it:

\begin{theorem}
\label{thm:intro_constructible}
Let $S_1,\dots,S_n\subseteq\Sval(K)$ pairwise disjoint.
    \begin{enumerate}
        \item[$\condU$] Assume that there exists a common uniformizer $\pi\in K$ of all $v\in S_1\cup\dots\cup S_n$.
        \item[$\condT$] Assume that each $S_i$ is compact in the Zariski topology on $\Sval(K)$.
        \item[$\condI$] Assume that 
        for any valuation $w$ on $K$ with a refinement in some $S_i$ and a refinement in some $S_j$, we have      
         $w(x_i-x_j) \geq w(z_i) = w(z_j)$.
    \end{enumerate}
Then there exists $x\in K$ with
$$
    v(x-x_i)>v(z_i)\mbox{ for all }v\in S_i,\mbox{  for }i=1,\dots,n.
$$
\end{theorem}

\begin{theorem}\label{thm:intro_Zariski_dual_closed}
Let $S_1,\dots,S_n\subseteq\Sval(K)$, not necessarily disjoint.
    \begin{enumerate}
        \item[$\condU$] Assume there exists a monic polynomial $f\in K[X]$ 
        such that for every $v\in  S_1\cup\dots\cup S_n$, $f$ has coefficients in the valuation ring of $v$
        and the reduction of $f$ has no zero in the residue field of $v$.
        \item[$\condT$] Assume that each $S_i$ is closed in the Hochster dual of the Zariski topology on $\Sval(K)$.
        \item[$\condI$]  Assume that for any $i$, $j$ and $w\in S_i\cap S_j$, 
        we have $w(x_i-x_j) \geq w(z_i) = w(z_j)$.
    \end{enumerate}
Then there exists $x\in K$ with
$$
v(x-x_i)\geq v(z_i)\mbox{ for all }v\in S_i,\mbox{  for }i=1,\dots,n.
$$
\end{theorem}

\begin{theorem}\label{thm:intro_orderings_compatible}
Let $S_1,\dots,S_n\subseteq\Sord(K)$ pairwise disjoint.
\begin{enumerate}
    \item[$\condT$] Assume that each $S_i$ is compact in the Harrison  topology on $\Sord(K)$.
    \item[$\condI$] Assume that if $w$ is a valuation on $K$ whose valuation ring $\mathcal{O}_w$ is convex both with respect to
an ordering in $S_i$ and with respect to an ordering in $S_j$, $i\neq j$, then these two orderings induce distinct orderings on the residue field of $w$,
and $w(x_i-x_j) \geq w(z_i) = w(z_j)$.
\end{enumerate}
Then there exists $x \in K$ with 
$$
 (x-x_i)^2 < z_i^2 \text{ for all orderings in } S_i, \mbox{ for } i=1,\dots,n.
$$ 
\end{theorem}

Note that condition \ref{thm:intro_orderings_compatible}(I) is in particular always satisfied
when distinct orderings on $K$ induce distinct topologies,
and so this in particular generalizes Theorem \ref{thm:intro_Prestel}.
In Remark \ref{rem:thm_main}, we explain how to deduce Theorems
\ref{thm:intro_constructible},
\ref{thm:intro_Zariski_dual_closed}
and \ref{thm:intro_orderings_compatible}
from Theorem \ref{thm:main},
which is proven in sections \ref{sec:balls}--\ref{sec:mainthm}.

In Section \ref{sec:value} we deduce from Theorem \ref{thm:main} some related results, namely value approximation theorems and residue approximation theorems.
In the literature such results are often needed on the way to prove an approximation theorem, but we deduce them as corollaries of the main theorem.
We also discuss special cases and applications
to $p$-valuations
and the connection to the strong approximation theorem in global fields.

In Section \ref{sec:exceptional_balls}
we extend Theorem \ref{thm:main} even further by allowing finitely many
exceptional valuations or (possibly complex) absolute values
for which condition $\condU$ might in general not be satisfiable,
thus obtaining a result that in addition contains
Theorem \ref{thm:artin_whaples} and Theorem \ref{thm:Ribenboim}.

Finally, in Section \ref{sec:functions} we use our approximation theorems to discuss a related problem, namely approximation of the values of rational functions.
The results of Sections \ref{sec:value} and \ref{sec:functions} are crucial ingredients of our paper  \cite{ADF3},
and the results of Section \ref{sec:mainthm} are used in \cite{ADF2}.

\section{The space of localities}
\label{sec:balls}
\label{sec:topologies}

We start by setting up the language unifying valuations and orderings in which we will phrase and prove the main theorem.
For basics on valuations and orderings we refer to \cite{EP} and \cite{Efrat}.

Let $K$ be a field.
We denote by $\B(K)$ the set of subsets $\mathcal{O}\subseteq K$
that satisfy
$\mathcal{O}\cdot\mathcal{O}\subseteq\mathcal{O}$ and $\mathcal{O}\cup\mathcal{O}^{-1}=K$
(where for $X\subseteq K$ we write $X^{-1}:=\{x^{-1}:0\neq x\in X\}$).

\begin{example}\label{ex:balls}
\begin{enumerate}
 \item $K\in\B(K)$
    \item If $|.|$ is an absolute value on $K$, then $\{x\in K:|x|\leq 1\}\in\B(K)$.
    \item If $v$ is a valuation on $K$ with valuation ring $\mathcal{O}_v$, then $\mathcal{O}_v\in\B(K)$.
    \item If $\leq$ is an ordering on $K$, then $\{x\in K:-1\leq x\leq 1\}\in\B(K)$.
\end{enumerate}
\end{example}

For $\mathcal{O}\in\B(K)$ we write
$\mathcal{O}^\times=\mathcal{O}\cap\mathcal{O}^{-1}$ and $\mathfrak{m}_\mathcal{O}=\mathcal{O}\setminus\mathcal{O}^\times$.
The following basic properties are easily checked:
$1\in\mathcal{O}$, $0\in\mathfrak{m}_\mathcal{O}$,
$\mathcal{O}\mathcal{O}=\mathcal{O}$, $\mathfrak{m}_\mathcal{O}\mathcal{O}=\mathfrak{m}_\mathcal{O}$.
The collection of sets $x\mathcal{O}$, for $x\in K^{\times}$, is totally ordered by inclusion.

We equip $\B(K)$ with the {\em Zariski topology} $\mathcal{T}_{\rm Zar}$ with subbasis
$$
 \{ \mathcal{O}\in\B(K) : x \in\mathcal{O} \}, x\in K.
$$ 
The corresponding {\em constructible topology} $\mathcal{T}_{\rm con}$ by definition 
has subbasis
$$
 \{ \mathcal{O}\in\B(K): x\in\mathcal{O} \}, \{\mathcal{O}\in\B(K):x\in\mathfrak{m}_\mathcal{O}\}, x\in K.
$$
With the constructible topology, $\B(K)$ is a closed subspace of the Stone space $2^{K}$, so it is compact and ${\rm T}_{2}$.
Noting also that the Zariski topology is ${\rm T}_{0}$, it follows from
\cite[Proposition~7]{Hoc69} that the Zariski topology is spectral,
and it follows from \cite[Proposition 8]{Hoc69} that the ``Hochster dual'' $\mathcal{T}_{\rm Zar^*}$ of the Zariski topology with subbasis%
\footnote{By definition, $\mathcal{T}_{\rm Zar^{*}}$ has as a basis for the closed sets the compact open sets in $\mathcal{T}_{\rm Zar}$,
i.e.~the finite unions of finite intersections of sets of the form $\{\mathcal{O}\in\B(K) : x\in\mathcal{O}\}$, for $x\in K$.
It follows that $\{\mathcal{O}\in\B(K) : x\in\mathcal{O}\}$, for $x\in K$, is a subbasis for the closed sets of $\mathcal{T}_{\rm Zar^{*}}$.
This is equivalent to the description given above.
}
$$
 \{ \mathcal{O}\in\B(K) : x \in\mathfrak{m}_\mathcal{O} \}, x\in K,
$$ 
is spectral, see also
\cite[\S1.4]{DST}.

\begin{remark}\label{rem:topologies}
Consider the following properties of a subset $X\subseteq\B(K)$:
\begin{enumerate}
\item $X$ is closed in $\mathcal{T}_{\rm Zar}$ or in $\mathcal{T}_{\rm Zar^*}$
\item $X$ is closed in $\mathcal{T}_{\rm con}$
\item $X$ is compact in $\mathcal{T}_{\rm con}$
\item $X$ is compact in $\mathcal{T}_{\rm Zar}$ and in $\mathcal{T}_{\rm Zar^*}$
\end{enumerate}
Then
$$
 (1) \Rightarrow (2) \Leftrightarrow (3) \Rightarrow (4).
$$
\end{remark}

We denote by $\Sval(K)$ the space of equivalence classes of valuations on $K$,
by $\Sord(K)$ the space of orderings on $K$, and by $\Sabsval(K)$ the space of equivalence classes of absolute values of $K$.
We write 
$$
 \Sall(K) = \Sval(K) \cup \Sord(K) \cup \Sabsval(K)
$$ 
for the union of these three spaces, but identify rank-$1$ valuations and archimedean orderings with their associated absolute values.
We call elements of $\Sall(K)$ {\em localities}\footnote{We deviate slightly from the terminology of \cite[Chapter 7]{Efrat}, 
where only elements of $\Sval(K) \cup \Sord(K)$ are called localities.} of $K$
and denote them by letters like $v$.
In the case $v\in\Sval(K)$, we denote by $v$ also a fixed valuation
in that equivalence class,
in the case $v\in\Sord(K)$ we write $\leq_v$ for the corresponding order relation on $K$, 
and in the case $v \in\Sabsval(K)$ we write $\lvert\cdot\rvert_v$ for a fixed absolute value in that equivalence class.
As usual, we call an absolute value $v\in\Sabsval(K)$ \emph{complex}
if $|.|_v$ is neither ultrametric nor induced by an archimedean ordering, i.e.\ if $v \in \Sabsval(K) \setminus (\Sval(K) \cup \Sord(K))$.

For $v\in\Sall(K)$ we write
$$
 \mathcal{O}_v = 
 \begin{cases} 
   \{x\in K:v(x)\geq0\}, & v\in\Sval(K),\\ 
   \{x\in K: -1\leq_v x\leq_v1\}, & v\in\Sord(K),\\
   \{x\in K: |x|_v\leq1\}, & v\in\Sabsval(K)\\
 \end{cases}
$$
and
$$
 \mathfrak{m}_v = 
 \begin{cases} 
   \{x\in K:v(x)>0\}, & v\in\Sval(K),\\ 
   \{x\in K: -1<_v x<_v1\}, & v\in\Sord(K),\\
   \{x\in K: |x|_v<1\}, & v\in\Sabsval(K).\\
 \end{cases}
$$
Note that these definitions agree if 
$v\in\Sabsval(K)\cap\Sval(K)$ or $v\in\Sabsval(K)\cap\Sord(K)$.
The map $v\mapsto\mathcal{O}_v$ gives an embedding of $\Sall(K)$ into $\B(K)$,
and we identify $\Sall(K)$ with its image.
Under this identification, the trivial valuation $v_{\mathrm{trivial}}$ is identified with $K\in\B(K)$,
and for any locality $v$ we have
$\mathfrak{m}_v=\mathfrak{m}_{\mathcal{O}_v}$.

\begin{remark}
\label{rem:Harrison}
On $\Sval(K)$, the Zariski topology induces the usual Zariski topology with subbasis
$$
 \{v\in\Sval(K) : v(x)\geq 0\}, x\in K,
$$
 and the constructible topology induces the usual constructible (or patch) topology with subbasis
$$
 \{v\in\Sval(K) : v(x)\geq 0\},\{v\in\Sval(K) : v(x)>0 \}, x\in K.
$$
Both the Zariski topology and the constructible topology induce
on $\Sord(K)$
the usual Harrison-topology (see for example \cite[VIII,\S6]{Lam05}) with subbasis
$$
 \{v\in\Sord(K) : x \geq_v 0 \}, x\in K,
$$
since for example $x\mapsto\frac{1-x}{1+x}$ exchanges
the intervals $(-1,1]$ and $[0,\infty)$.
In particular, $\Sval(K)$ and $\Sord(K)$ are compact and therefore closed in $\Sall(K)$ in the constructible topology,
while $\Sabsval(K)$ is in general not.
\end{remark}

\begin{example}\label{rem:topology.coincidence}
For $\pi\in K^{\times}$ and $e\in\mathbb{N}$ we denote by ${\rm S}_\pi^{e}(K)\subseteq\Sval(K)$ the set of valuations $v$ on $K$
with discrete value group $\Gamma_v$ (say $\mathbb{Z}$ is a convex subgroup of $\Gamma_v$) and $0 < v(\pi)\leq e$.
Then the three topologies $\mathcal{T}_{\mathrm{Zar}}$, $\mathcal{T}_{\mathrm{Zar}^{*}}$, and $\mathcal{T}_{\mathrm{con}}$
induce the same topology on ${\rm S}_{\pi}^{e}(K)$.
To see this, note that
$v(x)\geq0\Leftrightarrow v(x^{e}\pi)>0$,
and
$v(x)>0\Leftrightarrow v(x^{e}\pi^{-1})\geq0$,
for all $v\in {\rm S}_{\pi}^{e}(K)$ and all $x\in K^{\times}$.
This applies in particular to sets $S\subseteq{\rm S}_{\pi}^{1}(K)$ of valuations with a common uniformizer.

Moreover,
${\rm S}_{\pi}^{e}(K)$ is $\mathcal{T}_{\rm con}$-closed in $\Sval(K)$,
since
$v\in\Sval(K)$ is in ${\rm S}_{\pi}^{e}(K)$ if and only if
$v(\pi)>0$ and
for all $x\in K^{\times}$
we have
either
$v(x^{-1})\geq0$
or
$v(x^{e}\pi^{-1})\geq0$.
Since $\Sval(K)$ is compact in the constructible topology, it follows that ${\rm S}_{\pi}^{e}(K)$ is also compact.
\end{example}

Any locality $v\in\Sall(K)$ induces on $K$ a field topology, which we call the {\em $v$-topology}, defined by taking $\{z\mathfrak{m}_v:z\in K^\times\}$ as a basis for the filter of neighbourhoods of $0$.
The sets ${\rm B}_v(x,z):=x+z\mathfrak{m}_v$ with $x\in K$, $z\in K^\times$
form a basis for this topology.

If $\mathcal{O}_v\subseteq\mathcal{O}_w$, then $v$ is a {\em refinement} of $w$ and $w$ is a {\em coarsening} of $v$.
This defines a partial order on $\Sall(K)$,
and $v$ and $w$ are {\em incomparable} if they are incomparable in this partial order.

\begin{remark}
If $v,w$ are valuations then this terminology is standard;
whereas if $v$ is an ordering, then it has no proper refinement, and a valuation $w$ is a coarsening of $v$ if and only if $\mathcal{O}_{w}$ is convex with respect to $\leq_{v}$.
If $v \in \Sabsval(K) \setminus \Sval(K)$, i.e.\ $v$ is an archimedean absolute value (real or complex), then $v$ has no proper refinement, and its only proper coarsening is the trivial valuation, so in particular $v$ is incomparable to any other nontrivial element of $\Sall(K)$.
\end{remark}

\begin{remark}\label{rem:T1}
Note that by definition $v$ refines $w$ if and only if $v$ is in the Zariski closure of $\{w\}$, which is the case if and only if $w$ is in the $\mathcal{T}_{\rm Zar^*}$-closure of $\{v\}$.
In particular, a subspace of $\Sall(K)$ satisfies the ${\rm T}_1$ separation axiom with respect to either $\mathcal{T}_{\rm Zar}$ or $\mathcal{T}_{\rm Zar^*}$ if and only if its elements are pairwise incomparable.
\end{remark}

For every $v$ and $w$ in $\Sall(K)$ there exists a finest common coarsening $v\vee w$ in $\Sall(K)$.
If $w$ is a valuation coarsening $v$, then $v$ induces a locality $\bar{v}\in\Sall(Kw)$,
where $Kw$ denotes the residue field of $w$.
Moreover, either $v$ and $\bar{v}$ are both valuations, or both orderings, or both complex absolute values (and $w$ is trivial in this case).

Two localities $v,w\in\Sall(K)$ are {\em independent}
if they induce distinct topologies on $K$.
If both $v$ and $w$ are nontrivial,
this is the case if and only if $v\vee w$ is the trivial valuation.
Dependence of localities is an equivalence relation on $\Sall(K)$.
It follows from the fact that the set of coarsenings of a given locality is totally ordered that
for any finite number of pairwise dependent nontrivial localities there is another nontrivial locality which is a coarsening of all of them.

We call $v$ and $w$ {\em strongly incomparable} if $v$ and $w$ are incomparable and induce 
distinct (and then automatically independent) localities on the residue field of their finest common coarsening $v\vee w$.
Two sets $S_1, S_2\subseteq\Sall(K)$ are called {\em incomparable} (respectively {\em  strongly incomparable}, or {\em independent}) if any element of $S_1$ is {incomparable} (respectively {strongly incomparable} or {independent}) to/from any element of $S_2$.

\begin{remark}
Any two non-trivial independent localities are strongly incomparable.
For two valuations, even incomparability is sufficient for strong incomparability,
while the failure of strong incomparability for orderings is described by the Baer--Krull theorem \cite[Theorem 2.2.5]{EP}.
\end{remark}

We have the following approximation theorem for pairwise independent localities, generalizing Theorem \ref{thm:Bourbaki}.
\begin{theorem}\label{thm:pairwise_indep_fin_approx}
  Let $v_1, \dotsc, v_n \in \Sall(K)$ be all nontrivial, let $x_1, \dotsc, x_n \in K$ and $z_1, \dotsc, z_n \in K^\times$.
  \begin{enumerate}
    \item[$\condI$] Assume that $v_1,\dots,v_n$ are pairwise independent.
  \end{enumerate}
  Then there exists $x\in K$ with
  \[ x-x_i \in z_i \mathfrak{m}_{v_i} \text{ for }i=1,\dots,n. \]
\end{theorem}
\begin{proof}
If all $v_i$ are in $\Sval(K) \cup \Sabsval(K)$, then this is \cite[Corollary 4.2]{PZ}.
For any non-archimedean ordering $v_i$, consider its finest proper coarsening $\tilde{v}_i$, a nontrivial valuation. 
Since $z_i \mathfrak{m}_{\tilde{v}_i} \subseteq z_i \mathfrak{m}_{{v_i}}$, we can replace $v_i$ by $\tilde{v_i}$
and observe that the pairwise independence is preserved,
thereby reducing to the situation without non-archimedean orderings.
\end{proof}

The following is a generalization of a special case of Theorem \ref{thm:Ribenboim}:

\begin{proposition}\label{prop:RibenboimLocalities}
  Let $v_1,v_2\in\Sall(K)$ and let $z_1, z_2 \in K^\times$ such that for the finest common coarsening $w=v_1 \vee v_2$ we have $z_1 \mathcal{O}_{w} = z_2 \mathcal{O}_{w}$.
  \begin{enumerate}
     \item If $v_1$ and $v_2$ are strongly incomparable, 
      then there exists $z \in K^\times$ with $z \in z_1 \mathfrak{m}_{v_1}$ and $z^{-1}\in z_2^{-1} \mathfrak{m}_{v_2}$.

      \item If $v_1$ and $v_2$ are comparable or strongly incomparable (e.g.~at least one of them is a valuation), 
     then there exists $z \in K^\times$ with $z\in z_1 \mathcal{O}_{v_1}$ and $z^{-1}\in z_2^{-1} \mathcal{O}_{v_2}$.
   \end{enumerate}
\end{proposition}

\begin{proof}
  Assume without loss of generality that $z_1 = 1$.
  Note that then in particular $z_1,z_2\in\mathcal{O}_w^\times$.

  In case (1), $w$ is a valuation and $v_1$ and $v_2$ induce independent $\bar{v}_1,\bar{v}_2\in\Sall(Kw)$.
  The nonempty set $\{ \bar{z} \in Kw^\times \colon \bar{z}^{-1} \in \bar{z}_2^{-1}\mathfrak{m}_{\bar{v}_2}\}$ is open with respect to the $\bar{v}_2$-topology on $Kw$ and hence contains a ball $\bar{x} + \bar{z}_2'\mathfrak{m}_{\bar{v}_2}$ with $\bar{x}, \bar{z}_2' \in Kw^\times$.
  By Theorem~\ref{thm:pairwise_indep_fin_approx} 
  there exists $\bar{z}\in Kw$ with $\bar{z}\in \bar{z}_1\mathfrak{m}_{\bar{v}_1}$ and $\bar{z} \in \bar{x} + \bar{z}_2'\mathfrak{m}_{\bar{v}_2}$, so $\bar{z}^{-1}\in \bar{z}_2^{-1}\mathfrak{m}_{\bar{v}_2}$.
  If $z\in\mathcal{O}_w$ is any lift of $\bar{z}$, then 
   $z\in z_1\mathfrak{m}_{v_1}$ and $z^{-1}\in z_2^{-1}\mathfrak{m}_{v_2}$.
   
   In case (2), if $v_1$ and $v_2$ are comparable, say $\mathcal{O}_{v_1}\subseteq\mathcal{O}_{v_2}$,
   then $w=v_2$ and thus $z_1\mathcal{O}_{v_2}=z_2\mathcal{O}_{v_2}$, so $z=z_1$ satisfies the claim in this case.
   Otherwise $v_1$ and $v_2$ are strongly incomparable and 
 the claim then follows from the stronger claim (1).     
\end{proof}

\begin{remark}[Shifting and scaling]
\label{rem:shifting_and_scaling}
The assumption that $z_1=1$ in the preceding proof is an example of two general simplification principles which will be used in our proofs in various places.
Consider an approximation problem on two sets of localities $S_1$, $S_2$, where we wish to find an element $x \in K$ such that $x-x_1 \in z_1 \mathcal{O}_v$ for all $v \in S_1$, and $x-x_2 \in z_2 \mathcal{O}_v$ for all $v \in S_2$, for some given elements $x_1, x_2 \in K$ and $z_1, z_2 \in K^\times$.
If $x_1 = x_2$, then we have the trivial solution $x = x_1$, so we may assume $x_1 \neq x_2$.
For any constant $c \in K$, if we have a solution $x'$ for the modified approximation problem determined by $x_1' = x_1-c$ and $x_2'=x_2-c$, then $x=x'+c$ is a solution for the original approximation problem.
This \emph{shifting} of the problem can for instance be used to assume $x_1=0$ without loss of generality.
Similarly, for any constant $d \in K^\times$, if there is a solution $x'$ to the modified approximation problem given by $x_1'=dx_1$, $x_2'=dx_2$, $z_1'=dz_1$, $z_2'=dz_2$, then $x=d^{-1}x'$ is a solution to the original approximation problem.
This \emph{scaling}, together with previous shifting, can be used to reduce to the case $x_1=0$, $x_2=1$; alternatively (but usually not additionally), we can suppose that $z_1=1$, as in the proof of Proposition \ref{prop:RibenboimLocalities}.
The compatibility conditions (I) we impose in our theorems, see for instance \ref{thm:intro_constructible}(I), are all unaffected by scaling and shifting.
\end{remark}

\section{A uniformity condition on a set of localities}
\label{sec:constructible}

We now describe the general $\condU$ condition that we use in the main theorem and deduce some first consequences.
Fix a field $K$ and let $S\subseteq\Sall(K)$.
Consider the following assumption on $S$ and an element $t \in K^\times$.

\begin{assumption}\label{assumption:f_and_t}
  There exists a  polynomial $f \in K[X]$ of degree $d\geq 2$ 
  with leading coefficient $a_d$ 
   such that the following conditions are satisfied for all $v \in S$:
  \begin{enumerate}[(i)]
  \item\label{condIntegral}
    For any $x \in \mathcal{O}_v$ we have $f(x) \in \mathcal{O}_v$.
  \item\label{condNoApproxZero}
    For any $x \in \mathcal{O}_v$ we have $f(x) \not\in t \mathfrak{m}_v$.
  \item\label{condLeadingTerm}
    For any $x \not\in\mathcal{O}_v$ we have $f(x) \in x^d(\mathcal{O}_v \setminus t\mathfrak{m}_v) \cap (a_dx^d + x^{d-1}\mathcal{O}_v)$.
    \item\label{condPlus} We have $\mathcal{O}_v + \mathcal{O}_v \subseteq t^{-1} \mathcal{O}_v$. 
  \end{enumerate}
\end{assumption}

\begin{remark}\label{rem:ad_in_Ov}
Note that 
conditions
(\ref{condIntegral}, \ref{condNoApproxZero})
as well as (\ref{condLeadingTerm})
imply
that $t \in \mathcal{O}_v$
and therefore condition 
(\ref{condPlus})
is trivial if $v$ is a valuation.
Moreover, conditions (\ref{condNoApproxZero}, \ref{condLeadingTerm})
give that $f$ has no zero in $K$ if $S$ is nonempty.
If $v$ is a complex absolute value, then $K$ is dense in its completion $\hat K \cong \mathbb{C}$; since $f$ has a zero in $\hat K$, conditions (\ref{condNoApproxZero},\ref{condLeadingTerm}) will never both be satisfied in that case.
We will therefore mostly ignore complex absolute values until Section \ref{sec:exceptional_balls}.

Condition (\ref{condLeadingTerm}) implies that $a_d \in \mathcal{O}_v$: If $v$ is a valuation and $a_d \not\in\mathcal{O}_v$, then for any $x \not\in \mathcal{O}_v$ we have $x^d \mathcal{O}_v \cap (a_dx^d + x^{d-1} \mathcal{O}_v) = \emptyset$, since a member of the intersection would have to have valuation both equal to $v(a_dx^d)$ and at least $v(x^d)$.
If $v$ is an ordering (and very similar if $v$ is a complex absolute value) and $a_d \not\in \mathcal{O}_v$, then for $x >_v \max(1, (\lvert a_d \rvert - 1)^{-1})$ we have $x^d \mathcal{O}_v + x^{d-1} \mathcal{O}_v \subseteq x^d (1+x^{-1}) \mathcal{O}_v$, and all elements of this set are of smaller absolute value than $a_d x^d$, so again $x^d \mathcal{O}_v \cap (a_dx^d + x^{d-1} \mathcal{O}_v) = \emptyset$.
\end{remark}

\begin{remark}
Note that Assumption \ref{assumption:f_and_t}
remains true when $S$ is enlarged by adding the coarsenings of all localities in $S$; we may thus assume that $S$ is closed under coarsenings.
The assumption also remains true when passing to a subset of $S$.
\end{remark}

\begin{remark}\label{rem:t_1_valuation}
In the situation of Assumption
 \ref{assumption:f_and_t}, if $t \in \mathcal{O}_v^\times$ for some $v \in S$, then conditions (\ref{condIntegral},\ref{condNoApproxZero}) imply that $v$ is necessarily a valuation.
 In particular, if $t=1$, then $S$ consists only of valuations.
\end{remark}

\begin{example}\label{ex:V}
Assumption \ref{assumption:f_and_t} is satisfied for the following sets $S\subseteq\Sall(K)$ and elements $t \in K^\times$:
\begin{enumerate}
  \item If $S\subseteq\Sval(K)$ and there exists a monic $f \in K[X]$ such that for all valuations $v \in S$ the coefficients of $f$ are in the valuation ring of $v$ and the reduction of $f$  does not have a zero in the residue field of $v$, then the assumption is satisfied with $t = 1$.
\item If $S\subseteq{\rm S}_{\pi}^{1}(K)$ is a set of valuations with a common uniformizer $\pi \in K^\times$
(cf.~Example \ref{rem:topology.coincidence}), then the assumption is satisfied with $t = \pi$, as we may consider $f = X^2 - \pi$. 
  Slightly more generally, assume that $S\subseteq{\rm S}_{\pi}^{e}(K)$, for some $\pi\in K^\times$ and $e>0$ (also cf.\ Example \ref{rem:topology.coincidence}).
  Then the assumption is satisfied with $t = \pi$, by choosing $f = X^{e+1} - \pi$. 
  \item If $S\subseteq\Sord(K)$ then the assumption is satisfied with $t = \frac{1}{2}$, by choosing $f=\frac{1}{2}(X^2+1)$.
\end{enumerate}
\end{example}

\begin{example}
\label{ex:p-vals_plus_orderings}
Assume that $K$ is of characteristic zero.
For a prime number $p$ and $e\in\mathbb{N}$,
we note that ${\rm S}_p^e(K)$ contains all {\em $p$-valuations} in the sense of \cite{PR}
with $p$-ramification index at most $e$ (and arbitrary residue degree).
Then for every finite set of prime numbers $\mathcal{P}$ with product $q$ and $e\in\mathbb{N}$,
the space 
$$
 S=\Sord(K)\cup\bigcup_{p\in\mathcal{P}}{\rm S}_p^e(K)
 $$ 
 satisfies Assumption \ref{assumption:f_and_t} with $t=\frac{q}{2q+1}$, as is seen by taking $f = \frac{q+1}{2q+1}(X^{2e} + \frac{q}{q+1})$.
 In particular, in the language of \cite{Feh13},
 if $S_0$ is a finite set of primes of a number field $K_0$
 and $K$ is an extension of $K_0$, then for any type $\tau \in \mathbb{N}^2$ the set $\mathcal{S}_{S_0}^\tau(K)$ and this choice of $t$ satisfy Assumption \ref{assumption:f_and_t}.
\end{example}

\begin{example}
\label{ex:V_archimedean}
    As a different criterion, let $S \subseteq \Sord(K)$,  $f = \sum_{i=0}^d a_i X^i \in \mathbb{Q}[X]$ and $t \in \mathbb{Q}^\times$.
    If
    \[ \sum_{i=0}^d \lvert a_i \rvert \leq 1, \quad \lvert t \rvert + \sum_{i=0}^{d-1} \lvert a_i\rvert \leq \lvert a_d \rvert, \quad \text{and}\quad  \lvert t \rvert \leq \frac{1}{2}, \]
    then conditions \ref{assumption:f_and_t}(\ref{condIntegral},\ref{condLeadingTerm},\ref{condPlus}) are satisfied for all $v \in S$.
    If additionally $\lvert f(x) \rvert \geq \lvert t \rvert$ for every $x \in \mathbb{R}$, then the same holds for every real-closed field and therefore for every ordered field, so condition \ref{assumption:f_and_t}(\ref{condNoApproxZero}) is also satisfied for all $v \in S$.
\end{example}

\begin{example}\label{ex:orderings_and_residue_fields}
  Let $g \in \mathbb{Z}[X]$ be a monic polynomial with no real zero, so in particular $g$ as a function $\mathbb{R} \to \mathbb{R}$ is bounded away from zero. Write $g = \sum_{i=0}^d a_i X^i$.
  For a small rational number $b > 0$, consider the polynomial 
  $$
   f = b^{d+1} g(b^{-1}X) = b\cdot \sum_{i=0}^d a_ib^{d-i} X^i \in \mathbb{Q}[X]
   $$ and $t = b^{d+2}$.
  For sufficiently small $b$, this clearly satisfies all the conditions from Example \ref{ex:V_archimedean}.
  We may in fact choose $b = 1/l$ for some large prime number $l$.
  
  Now let $K$ be any field of characeristic not $l$, and consider $f$ and $t$ as above.
  In this situation Assumption \ref{assumption:f_and_t} is satisfied with $S=\Sord(K) \cup S'$ and $t$ as above, where $S'\subseteq\Sval(K)$ is the set of valuations on $K$ with residue characteristic not $l$ in whose residue field the reduction of $g$ does not have a zero:
  For any valuation $v \in S'$, the polynomial $f$ has all coefficients in $\mathcal{O}_v$, its reduction has no zero over the residue field since the reduction of $g$ does not, and $t \in \mathcal{O}_v^\times$.
  For orderings $v \in \Sord(K)$, we have forced the assumption to be satisfied by construction.

  Note that by the Chebotarev Density Theorem (see \cite[VII, (13.4)]{Neu}), the set $S'$ will include all valuations on $K$ with residue field $\mathbb{F}_p$ for some positive density of prime numbers $p$.
  By choosing $g$ suitably, e.g.\ as the minimal polynomial of an integral primitive element of a totally imaginary Galois extension $L/\mathbb{Q}$, one can in fact achieve densities arbitrarily close to $1$, because the reduction of the polynomial $g$ has no root in $\mathbb{F}_p$ for all prime numbers $p$ such that $p$ is neither completely split nor ramified in $L/\mathbb{Q}$, cf.~\cite[I, (8.3)]{Neu}.
\end{example}

\begin{proposition}\label{prop:phiminmax}
  Let $t$ and $f$ be given.
There exists a function
$\phi:K^\times\times K^\times\longrightarrow K^\times$
such that for all $v\in{\rm S}(K)$ for which conditions \ref{assumption:f_and_t}(\ref{condIntegral},\ref{condNoApproxZero}, \ref{condLeadingTerm}) are satisfied,
and for all $x,y\in K^\times$,
the following hold:
\begin{enumerate}
\item
If $x\in t\mathcal{O}_v$ and $y\in t\mathcal{O}_v$,
then $\phi(x,y)\in x\mathcal{O}_v\cup y\mathcal{O}_v$.

\item
If $x^{-1}\in\mathcal{O}_v$
or $y^{-1}\in\mathcal{O}_v$,
then $\phi(x,y)^{-1}\in x^{-1}\mathcal{O}_v\cap y^{-1}\mathcal{O}_v$.

\end{enumerate}
\end{proposition}
\begin{proof}
Let $\phi$ be the homogenisation of $ft^{-1}$,
so if $f$ is of degree $d\geq 2$ then $\phi$ is given by
\begin{align*}
    \phi(X,Y)=f(XY^{-1})Y^{d}t^{-1} \in K[X,Y].
\end{align*}
Let $v\in{\rm S}(K)$ satisfy
conditions \ref{assumption:f_and_t}(\ref{condIntegral},\ref{condNoApproxZero}, \ref{condLeadingTerm}),
and write $\mathcal{O}=\mathcal{O}_v$, $\mathfrak{m}=\mathfrak{m}_v$.

To prove (1),
let $x,y\in t\mathcal{O}$.
Then $x^{d}t^{-1}\in x\mathcal{O}$ and $y^{d}t^{-1}\in y\mathcal{O}$.
If $xy^{-1}\in\mathcal{O}$,
then
$f(xy^{-1})\in\mathcal{O}$,
by \ref{assumption:f_and_t}(\ref{condIntegral}),
hence
\begin{align*}
    \phi(x,y)=f(xy^{-1})y^{d}t^{-1}\in y\mathcal{O}.
\end{align*}
If instead
$xy^{-1}\notin\mathcal{O}$,
then
$f(xy^{-1})\in(xy^{-1})^{d}\mathcal{O}$ by \ref{assumption:f_and_t}(\ref{condLeadingTerm}),
hence
\begin{align*}
    \phi(x,y)&=f(xy^{-1})y^{d}t^{-1}
    \in (xy^{-1})^{d}y^{d}t^{-1}\mathcal{O}
    = x^{d}t^{-1}\mathcal{O}
    \subseteq x\mathcal{O},
\end{align*}
as required.

To prove (2),
our assumption is that either $x^{-1}\in\mathcal{O}$ or $y^{-1}\in\mathcal{O}$.
If $xy^{-1}\in\mathcal{O}$,
then
$y^{-1}\mathcal{O}\subseteq x^{-1}\mathcal{O}$,
so the assumption implies that $y^{-1}\in\mathcal{O}$.
From \ref{assumption:f_and_t}(\ref{condNoApproxZero}) we have $f(xy^{-1})\notin t\mathfrak{m}$,
so $f(xy^{-1})^{-1}\in t^{-1}\mathcal{O}$.
Thus
\begin{align*}
    \phi(x,y)^{-1}&=f(xy^{-1})^{-1}y^{-d}t
    \in y^{-d}\mathcal{O}
    \subseteq y^{-1}\mathcal{O}=x^{-1}\mathcal{O}\cap y^{-1}\mathcal{O}.
\end{align*}
On the other hand,
if $xy^{-1}\notin\mathcal{O}$,
then
$x^{-1}y\in\mathcal{O}$ and  $x^{-1}\mathcal{O}\subseteq y^{-1}\mathcal{O}$,
and the assumption implies that $x^{-1}\in\mathcal{O}$.
By \ref{assumption:f_and_t}(\ref{condLeadingTerm}), $f(xy^{-1})\notin(xy^{-1})^dt\mathfrak{m}$, so
$f(xy^{-1})^{-1}\in (x^{-1}y)^dt^{-1}\mathcal{O}$.
Thus
\begin{align*}
    \phi(x,y)^{-1}&=f(xy^{-1})^{-1}y^{-d}t
    \in (x^{-1}y)^{d}t^{-1}y^{-d}t\mathcal{O}
    = x^{-d}\mathcal{O}
    \subseteq x^{-1}\mathcal{O}=x^{-1}\mathcal{O}\cap y^{-1}\mathcal{O}.\qedhere
\end{align*}    
\end{proof}

\begin{corollary}\label{cor:phimin}
  Let $t$ and $f$ be given, and let $\phi$ be as in Proposition \ref{prop:phiminmax}.
  Let $v\in{\rm S}(K)$ satisfy conditions \ref{assumption:f_and_t}(\ref{condIntegral},\ref{condNoApproxZero}, \ref{condLeadingTerm}).
  Then for $x_1, \dotsc, x_n \in K^\times$, the non-zero element 
  $$
   \phi(x_1, \dotsc, x_n) := \left\{\begin{array}{ll}x_{1},&n=1\\ \phi(x_1, \phi(x_2,\dots)\dotsb), &n>1\end{array}\right.
  $$ 
  satisfies:
  \begin{enumerate}
    \item If $x_i \in t\mathcal{O}_v$ for all $i$, then $\phi(x_1, \dotsc, x_n) \in \bigcup_{i=1}^n x_i \mathcal{O}_v$.
    \item If $x_i^{-1} \in \mathcal{O}_v$ for some $i$, then $\phi(x_1, \dotsc, x_n)^{-1} \in \bigcap_{i=1}^n x_i^{-1} \mathcal{O}_v$.
  \end{enumerate}
\end{corollary}
\begin{proof}
  For (1), induction using Proposition \ref{prop:phiminmax}(1) firstly shows that for any $i \geq 1$, the element $\phi(x_i, \dotsc, x_n)$ is in $t\mathcal{O}_v$.
  Secondly, using Proposition \ref{prop:phiminmax}(1) again, we obtain \[\phi(x_1, \dotsc, x_n) = \phi(x_1, \phi(x_2, \dotsc, x_n)) \in x_1 \mathcal{O}_v \cup \phi(x_2, \dotsc, x_n) \mathcal{O}_v,\] so using induction once more gives $\phi(x_1, \dotsc, x_n) \in \bigcup_i x_i \mathcal{O}_v$ as desired.

  For (2) note that it is sufficient to prove that $\phi(x_1, \dotsc, x_n)^{-1} \in x_i^{-1} \mathcal{O}_v$ for those $i$ with $x_i^{-1} \in \mathcal{O}_v$, since for $x_i^{-1} \in \mathcal{O}_v$ and $x_{i'}^{-1} \notin \mathcal{O}_v$ we have $x_i^{-1} \mathcal{O}_v \subseteq x_{i'}^{-1}\mathcal{O}_v$.
  For any such $i$, we have $\phi(x_i, \dotsc, x_n)^{-1} \in x_i^{-1} \mathcal{O}_v$ by Proposition \ref{prop:phiminmax}(2), and then induction shows $\phi(x_j, \dotsc, x_n)^{-1} \in x_i^{-1} \mathcal{O}_v$ for all $j \leq i$, so we obtain the claim $\phi(x_1, \dotsc, x_n)^{-1} \in x_i^{-1}  \mathcal{O}_v$.
\end{proof}

\begin{remark}
Only for use in Section \ref{sec:exceptional_balls}
we note that the proof of part (1) of Proposition \ref{prop:phiminmax},
and consequently also of part (1) of Corollary \ref{cor:phimin},
requires only conditions \ref{assumption:f_and_t}(\ref{condIntegral})
and \ref{assumption:f_and_t}(\ref{condLeadingTerm}) to hold.
\end{remark}

\begin{remark}\label{rem:alternative.phi}
  A function $\phi$ satisfying the conditions from Proposition \ref{prop:phiminmax} can also be obtained using assumptions different from our Assumption \ref{assumption:f_and_t}.
  For instance, if one works with a set $S$ consisting exclusively of valuations such that the holomorphy ring $R = \bigcap_{v \in S} \mathcal{O}_v$ is a Bézout ring with quotient field $K$, then for any $x, y \in K^\times$ we can find $z \in K^\times$ with $v(z) = \min(v(x), v(y))$ for all $v \in S$: this is exactly the statement that the fractional ideal $xR + yR$ is principal with generator $z$.
  This $z$ will satisfy the properties desired from $\phi(x,y)$ in Proposition \ref{prop:phiminmax}.

  As a slightly more general condition, it is sufficient for $R$ to be a Prüfer ring such that there exists a natural number $n$ for which the $n$-th power of any two-generated fractional ideal of $R$ is principal: In this situation, for any $x,y \in K^\times$, we can find a generator $z$ of the fractional ideal $(xR + yR)^n$, and this generator $z$ will satisfy $v(z) = n \min(v(x), v(y))$ for all $v \in S$, which is as desired for $\phi(x,y)$.
  This latter condition on $R$ is established in \cite[Theorem 1]{RoquettePrincIdealThms} in the situation of Example \ref{ex:V}(1), using effectively our construction of the function $\phi$.
  
  Ring-theoretic conditions on the holomorphy ring are in the style of the \condU{} conditions of 
results of Ershov, see Section \ref{sec:Ershov}.
  Our Assumption \ref{assumption:f_and_t} is more ad hoc, but has the advantage that it can also apply to orderings.
\end{remark}

\section{The approximation theorem}
\label{sec:mainthm}

We want to prove two approximation theorems:
one with a conclusion of the type $x - x_i \in z_i \mathfrak{m}$ and one with the weaker conclusion $x - x_i \in z_i \mathcal{O}$.
Neither of these theorems will easily imply the other, but the proofs are very similar, so we state and prove our lemmas simultaneously in the two situations, called ``Situation $\mathfrak{m}$'' and ``Situation $\mathcal{O}$''.

Let $S\subseteq\Sall(K)$.
For $v \in \Sall(K)$, we write $\BB_v = \mathfrak{m}_v$ in Situation~$\mathfrak{m}$, and $\BB_v = \mathcal{O}_v$ in Situation~$\mathcal{O}$.
Moreover, in Situation $\mathfrak{m}$, we endow $\Sall(K)$ with the topology $\mathcal{T}_{\rm Zar^*}$,
in Situation $\mathcal{O}$ with the Zariski topology $\mathcal{T}_{\rm Zar}$. 
In this way, in either case sets of the form $\{ v\in\Sall(K) \colon x \in \BB_v \}$, for $x \in K$, are open.

Let us also fix $t \in K^\times$.
We call two localities $v,w$ \emph{$t$-independent} if $t \in \mathcal{O}_{v \vee w}^\times$ and additionally $v,w$ are strongly incomparable if they are both orderings. 
Note that any two independent localities are always $t$-independent, and any two valuations (not necessarily distinct) are $1$-independent.
We call two sets $S_1,S_2$ of localities
{\em $t$-independent}
if any element of $S_1$ is $t$-independent to any element of $S_2$.
\begin{example}\label{ex:tincomp}
Let $t=\pi\in K$ and $e\in\mathbb{N}$. Valuations $v,w\in{\rm S}_\pi^e(K)$
(cf.~Example \ref{rem:topology.coincidence}),
are $t$-independent if and only if
they are distinct.
\end{example}

For the rest of this section
we assume that $S$ and $t$ satisfy
Assumption \ref{assumption:f_and_t}.
Indeed, 
we fix a polynomial $f\in K[X]$
as described there.

We now start the proof of the approximation theorem
by constructing elements that are `uniformly small'
resp.~`uniformly big'.

\begin{lemma}\label{lem:pairwise_indicator_functions}
  Let $S_1, S_2 \subseteq S$ be nonempty and compact in the given topology. Assume they are $t$-independent, and in Situation $\mathfrak{m}$ furthermore incomparable.
  Let $z, z' \in K^\times$ such that $w(z) \leq 0$ and $w(z') \leq 0$ whenever $w$ is a valuation on $K$ which has a refinement in $S_1$ and a refinement in $S_2$.
  Then there exists $b \in K^\times$ such that $b \in t \BB_v \cap tz \BB_v$ for all $v \in S_1$, and $b^{-1} \in t \BB_v \cap tz' \BB_v$ for all $v \in S_2$.
\end{lemma}
\begin{proof}
For fixed $v\in S_{1}$ and $v'\in S_{2}$, denote $w=v\vee v'$. 
This is a valuation by the assumption of $t$-independence,
and so $w(z),w(z')\leq0$, and $w(t)=0$.

With the aim of applying Proposition \ref{prop:RibenboimLocalities},
we define two elements $z_v,z_{v'}'$ as modifications of $z,z'$, as follows.
If $z\in\mathcal{O}_{v}$ let $z_v=tz$, otherwise let $z_v=t$.
If $z'\in\mathcal{O}_{v'}$ let $z_{v'}'=(tz')^{-1}$, otherwise let $z_{v'}' = t^{-1}$.
Note that $w(z_{v})=w(z_{v'}')=0$.
Therefore $z_v\mathcal{O}_w=\mathcal{O}_w=z_{v'}'\mathcal{O}_w$,
so
the assumption of Proposition \ref{prop:RibenboimLocalities} is satisfied.

In Situation $\mathfrak{m}$ we have that $v$ and $v'$ are strongly incomparable and
can apply Proposition \ref{prop:RibenboimLocalities}(1) to get
$y\in K^\times$ with $y\in z_v\mathfrak{m}_v=z_v \BB_v$ and $y^{-1}\in z_{v'}'^{-1}\mathfrak{m}_{v'} =z_{v'}'^{-1} \BB_{v'}$;
In Situation $\mathcal{O}$ we can apply Proposition \ref{prop:RibenboimLocalities}(2) to get
$y\in K^\times$ with 
$y\in z_v\mathcal{O}_v=z_v\BB_v$ and $y^{-1}\in z_{v'}'^{-1}\mathcal{O}_{v'}=z_{v'}'^{-1}\BB_{v'}$.
Thus, setting $b_{vv'}=y$, we get in both situations that
\begin{enumerate}
\item $b_{vv'} \in z_{v}\BB_{v}=t \BB_v \cap tz \BB_v$ and 
\item $b_{vv'}^{-1} \in z_{v'}'^{-1}\BB_{v'}=t \BB_{v'} \cap tz' \BB_{v'}$.
\end{enumerate}
For every $x\in K^\times$, the set
$$
U_x^{(2)} = \{u \in S: x^{-1} \in t \BB_u \cap tz' \BB_u \}=\{u\in S: (xt)^{-1}\in \BB_u\}\cap\{u\in S: (xtz')^{-1} \in \BB_u\}
$$
is open in the given topology.
For each $v\in S_{1}$,
the family $\{U_{b_{vv'}}^{(2)}:v'\in S_2\}$ is an open covering of $S_2$, by (2).
By compactness of $S_2$,  there are finitely many $v'_1, \dotsc, v'_m\in S_2$ such that 
$S_2\subseteq \bigcup_{i=1}^m U_{b_{vv'_i}}^{(2)}$.
Let 
$b_v = \phi(b_{vv'_1},\dots,b_{vv'_m})$.
For each $i$, by (1), we have $b_{vv_{i}'}\in t\BB_{v}\cap tz\BB_{v}\subseteq t\mathcal{O}_{v}$,
and so Corollary \ref{cor:phimin}(1) gives that
\begin{enumerate}
    \item[$(1')$]
    $b_{v}\in \bigcup_{i=1}^{m}b_{vv_{i}'}\mathcal{O}_{v}\subseteq t\BB_{v}\cap tz\BB_{v}$.
\end{enumerate}
On the other hand, for each $v' \in S_2$,
there is an $i$ with $v' \in U_{b_{vv'_i}}^{(2)}$,
i.e.\ $b_{vv'_i}^{-1} \in t \BB_{v'} \cap tz' \BB_{v'}\subseteq\mathcal{O}_{v'}$,
and so  Corollary \ref{cor:phimin}(2) gives that
\begin{enumerate}
\item[$(2')$]  $b_v^{-1}\in \bigcap_{i=1}^{m}b_{vv_{i}'}^{-1}\mathcal{O}_{v'}\subseteq t\BB_{v'} \cap tz' \BB_{v'}$.
\end{enumerate}

Similarly, for every $x\in K$, the set
$$
 U_x^{(1)}=\{u \in S: x\in t\BB_u\cap tz \BB_u\}
$$
is open,
and the family $\{U_{b_v}^{(1)}:v\in S_1\}$ is an open covering of $S_1$, by $(1')$.
By compactness of $S_1$, there exist finitely many $v_1, \dotsc, v_k\in S_1$ such that 
$S_1\subseteq\bigcup_{i=1}^k U_{b_{v_i}}^{(1)}$.
Let
$b = \phi(b_{v_1}^{-1}, \dots,b_{v_k}^{-1})^{-1}$.

For each $v\in S_1$, there is an $i$ with $v\in U_{b_{v_i}}^{(1)}$, i.e.\
$(b_{v_i}^{-1})^{-1}=b_{v_i} \in t\BB_v \cap tz \BB_v\subseteq\mathcal{O}_v$ and so Corollary \ref{cor:phimin}(2) gives that
\begin{enumerate}
\item[$(1'')$]
$b\in\bigcap_{i=1}^{k}b_{v_{i}}\mathcal{O}_{v}\subseteq t\BB_v\cap tz \BB_v$.
\end{enumerate}
Finally, for each $v'\in S_2$ and each $i$, by $(2')$,
we have that $b_{v_i}^{-1} \in t\BB_{v'} \cap tz'\BB_{v'}\subseteq t\mathcal{O}_{v'}$,
and so Corollary \ref{cor:phimin}(1) gives that
\begin{enumerate}
\item[$(2'')$]
$b^{-1}\in\bigcup_{i=1}^{k}b_{v_{i}}^{-1}\mathcal{O}_{v'}\subseteq t\BB_{v'} \cap tz'\BB_{v'}$,
\end{enumerate}
as required.
\end{proof}

\begin{lemma}\label{lem:approximation_lemma}
  Let $d$ be the degree of $f$ and $a_d$ its leading coefficient.
  Let $v \in S$, $z \in K^\times$ and $b \in K^\times$. Then the element $x = a_db^{-d}f(b^{-1})^{-1}$ satisfies 
\begin{enumerate}
\item $x-1 \in z \BB_v$ if $b \in t\BB_v \cap tz\BB_v$, and 
\item $x \in z\BB_v$ if $b^{-1} \in t\BB_v \cap tz\BB_v$.
\end{enumerate}
\end{lemma}
\begin{proof}
  Recall that $t\in\mathcal{O}_v$ and $a_d\in\mathcal{O}_v$ 
   (Remark \ref{rem:ad_in_Ov}).

 Proof of (1):
  Assume that $b \in t\BB_v \cap tz\BB_v$.
  If $b \in \mathfrak{m}_v$ 
  (which is the case if $t \in \mathfrak{m}_v$,
   $z\in\mathfrak{m}_v$,
   or $\BB_v=\mathfrak{m}_v$), 
  then $b^d f(b^{-1}) \in (\mathcal{O}_v \setminus t\mathfrak{m}_v) \cap (a_d + b \mathcal{O}_v)$
   by \ref{assumption:f_and_t}(\ref{condLeadingTerm}), so
   in particular
   $a_d-b^df(b^{-1})\in b\mathcal{O}_v$  and
   $(b^df(b^{-1}))^{-1}\in t^{-1}\mathcal{O}_v$.
   Therefore,
   \[x-1 = \frac{a_d - b^df(b^{-1})}{b^d f(b^{-1})} \in b\mathcal{O}_v t^{-1}\mathcal{O}_v = bt^{-1} \mathcal{O}_v \subseteq z \BB_v \] as desired.
   If $b \not\in\mathfrak{m}_v$, 
   then 
   $\BB_v=\mathcal{O}_v$, 
   $t\in\mathcal{O}_v^\times$ 
   and $z^{-1}\in\mathcal{O}_v$.
   Combining $b\in t\BB_v=\mathcal{O}_v$ and $b^{-1}\in\mathcal{O}_v$
    we then get $b\in\mathcal{O}_v^\times$.
   Therefore, $b^d f(b^{-1}) \in \mathcal{O}_v \setminus t\mathfrak{m}_v = \mathcal{O}_v^\times$ by \ref{assumption:f_and_t}(\ref{condIntegral},\ref{condNoApproxZero}), and therefore $x \in \mathcal{O}_v$.
 Since $t\in\mathcal{O}_v^\times$ implies that $v$ is a valuation (Remark \ref{rem:t_1_valuation}), we conclude that
 $$
  x-1 \in \mathcal{O}_v\subseteq z\mathcal{O}_v=z\BB_v.
 $$ 
 
 Proof of (2):
 Now assume that $b^{-1} \in t\BB_v \cap tz\BB_v$.
 We have $f(b^{-1}) \not\in t \mathfrak{m}_v$ by \ref{assumption:f_and_t}(\ref{condNoApproxZero}), so 
\begin{equation*}
  x \in a_db^{-d} t^{-1}\mathcal{O}_v \subseteq b^{-d+1} \BB_v \subseteq z\BB_v.
  \qedhere
\end{equation*}
\end{proof}

\begin{remark}\label{rem:Lemma52_details}
Again only 
for use in Section \ref{sec:exceptional_balls} we point out that the proof of 
Lemma \ref{lem:approximation_lemma}(1) in Situation $\mathfrak{m}$ requires only \ref{assumption:f_and_t}(\ref{condLeadingTerm}) to hold for $v$, $t$ and $f$, and not all the other conditions from Assumption \ref{assumption:f_and_t}.
\end{remark}

\begin{lemma}\label{lem:induction}
  Let $S_1, \dotsc, S_n \subseteq S$ be nonempty and compact in the given topology. Assume they are pairwise $t$-independent, and in Situation $\mathfrak{m}$ furthermore pairwise incomparable.
  Let $x', x'' \in K$, $z_1, \dotsc, z_n \in K^\times$ such that for any valuation $w$ on $K$ which has a refinement in $S_i$ and a refinement in $S_j$, $i \neq j$, we have $w(z_i) = w(z_j)$, and furthermore for any valuation $w$ with a refinement in $S_1$ and a refinement in $S_i$, $i \neq 1$, we have $w(x'-x'') \geq w(z_1)$.
 Then there exists $x\in K$ with $x-x'\in z_1 \BB_v$ for each $v\in S_1$
 and $x-x''\in z_i \BB_v$ for each $v\in S_i$, $i\neq 1$.
\end{lemma}
\begin{proof}
  If $x' = x''$, we may take $x= x'$, so assume this is not the case.
  By shifting and scaling as in Remark \ref{rem:shifting_and_scaling},
  we may assume that $x''=0$ and $x'=1$. (Observe that all the hypotheses and the claim are invariant under shifting and scaling.)
  In this situation, the compatibility hypothesis 
  implies that $w(z_1)=w(z_i)\leq 0$ for every $w$ with a refinement in $S_1$ and a refinement in $S_i$, $i\neq 1$.
  
  For any $i \neq 1$, Lemma \ref{lem:pairwise_indicator_functions} provides an element $b_i \in K^\times$ with 
  $b_i\in t\BB_v \cap tz_1\BB_v$ for $v \in S_1$ and $b_i^{-1} \in t \BB_v \cap tz_i \BB_v$ for $v \in S_i$.
  Let $b = \phi(b_2, \dotsc, b_n)$.
  By Corollary \ref{cor:phimin}(1,2), $b \in t\BB_v \cap tz_1\BB_v$ for $v \in S_1$, and $b^{-1} \in t\BB_v \cap tz_i\BB_v$ for $v \in S_i$, $i \neq 1$.
  Hence Lemma \ref{lem:approximation_lemma} finishes the proof.
\end{proof}

\begin{theorem}\label{thm:main}
Let $S_1,\dots,S_n\subseteq\Sall(K)$, $t \in K^\times$, $x_1,\dots,x_n\in K$ and $z_1,\dots,z_n\in K^\times$.
\begin{enumerate}
    \item[$\condU$] Assume that Assumption \ref{assumption:f_and_t}
    holds for $S=S_1\cup\dots\cup S_n$ and  $t$.
    \item[$\condT$] Assume each $S_i$ is compact in the given topology.
    \item[$\condI$] Assume that for any valuation $w$ on $K$ with a refinement in $S_i$ and a refinement in $S_j$ we have $w(x_i-x_j) \geq w(z_i) = w(z_j)$; assume further that the $S_i$ are pairwise $t$-independent, and in Situation $\mathfrak{m}$ furthermore pairwise incomparable.
\end{enumerate}
Then there exists $x \in K$ with 
$$
 x-x_i \in z_i\BB_v \text{ for all }v \in S_i, \mbox{ for } i=1,\dots,n.
$$ 
\end{theorem}

\begin{proof}
  We use induction on $n$.
  Without loss of generality assume that $S_i\neq\emptyset$ for all $i$.
  
  For $n = 1$ we may take $x = x_1$.
  
  For $n>1$, 
  we want to apply the induction hypothesis for the sets $S_2, \dotsc, S_n$,
  the elements $x_2,\dots,x_n$ and $tz_2,\dots,tz_n$.
  Note that, by the assumption of pairwise $t$-independence, $w(tz_i)=w(z_i)$ for every valuation $w$ with a refinement in $S_i$ and a refinement in $S_j$, $j\neq i$, 
  so the compatibility condition $w(x_i-x_j)\geq w(tz_i) = w(tz_j)$ is satisfied.
  The induction hypothesis thus
   gives an element $x'' \in K$ with 
  $x''-x_i \in z_i t\BB_v$ for all $v \in S_i$, $i \geq 2$.
  We now apply Lemma~\ref{lem:induction} to the sets $S_1,\dots,S_n$ and the elements
  $x': = x_1$, $x''$ and $z_1,z_2t,\dots,z_nt$.
  Note that if a valuation $w$ has a refinement  in $S_1$ and a refinement $v$ in $S_i$, $i\neq 1$, 
  then indeed 
  $$
   w(x'-x'')=w((x_1-x_i)+(x_i-x''))\geq w(z_{1}) = w(z_{i}),
  $$  
   as $x''-x_i\in z_it\BB_v\subseteq z_i\mathcal{O}_w$ implies that $w(x''-x_i)\geq w(z_i)$.
  We thus obtain $x\in K$ with $x-x_1 \in z_1\BB_v$ for all $v \in S_1$, 
  and $x - x'' \in z_it\BB_v$ for all $v \in S_i$, $i \neq 1$.
  Finally, for $v\in S_i$, $i\neq 1$, we get that
  $$
  x- x_i =(x-x'')+(x''-x_i)\in z_it\BB_v+z_it\BB_v\subseteq z_it\BB_v(\mathcal{O}_v+\mathcal{O}_v)\subseteq z_it\BB_v\cdot t^{-1}\mathcal{O}_v=z_i \BB_v
 $$ 
 by \ref{assumption:f_and_t}(\ref{condPlus}).
  This finishes the induction.
\end{proof}

\begin{remark}\label{rem:thm_main}
As promised, we now explain how to deduce Theorems \ref{thm:intro_constructible}, \ref{thm:intro_Zariski_dual_closed} 
and \ref{thm:intro_orderings_compatible} from the introduction.

Theorem \ref{thm:intro_constructible} is obtained 
by applying Theorem \ref{thm:main} in Situation~$\mathfrak{m}$ with $t=\pi$:
${\ref{thm:intro_constructible}}\condU$ implies ${\ref{thm:main}}\condU$ (see Example \ref{ex:V}(2)),
and ${\ref{thm:intro_constructible}}\condT$
together with ${\ref{thm:intro_constructible}}\condU$ implies
${\ref{thm:main}}\condT$
(see Example \ref{rem:topology.coincidence}).
The $S_i$ are pairwise $t$-independent (Example \ref{ex:tincomp})
and pairwise incomparable since pairwise disjoint,
so
${\ref{thm:intro_constructible}}\condI$ implies ${\ref{thm:main}}\condI$.

Theorem \ref{thm:intro_Zariski_dual_closed} is obtained by applying 
Theorem~\ref{thm:main} in Situation $\mathcal{O}$ with $t=1$:
${\ref{thm:intro_Zariski_dual_closed}}\condU$ implies ${\ref{thm:main}}\condU$ (see Example \ref{ex:V}(1)),
and
${\ref{thm:intro_Zariski_dual_closed}}\condT$ implies ${\ref{thm:main}}\condT$ (see Remark \ref{rem:topologies}).
In fact, since the $S_{i}$ are closed in the Hochster dual of the Zariski topology,
then the $S_{i}$ are also closed under coarsenings,
and so
${\ref{thm:intro_Zariski_dual_closed}}\condI$ implies ${\ref{thm:main}}\condI$.

Finally, Theorem \ref{thm:intro_orderings_compatible} 
follows from Theorem \ref{thm:main} in Situation $\mathfrak{m}$
using Example \ref{ex:V}(3)
and taking Remark~\ref{rem:Harrison} into consideration.
\end{remark}

\begin{remark}\label{T_needed}
Note that the set of $v\in\Sall(K)$ for which a given $x\in K$ 
satisfies an approximation condition 
$x-x_i \in z_i\BB_v$
is always open-closed in the constructible topology, in particular compact both in $\mathcal{T}_{\rm Zar}$ and in $\mathcal{T}_{{\rm Zar}^*}$.
This explains why condition \ref{thm:main}\condT{} is natural.
It is also clear that without this condition the theorem must fail.
\end{remark}

\begin{remark}\label{rem:condition_I}\label{I_needed}
It is obvious that assumption {\ref{thm:main}}\condI{} cannot simply be dropped.
It is also clear that 
if there exists $x$ with $x-x_i \in z_i\mathcal{O}_v$
for all $v \in S_i$, $i=1,\dots,n$,
then
any $w$ as in ${\ref{thm:main}}\condI$ must satisfy
$w(x_i-x_j)\geq\min\{w(z_i),w(z_j)\}$,
but the condition $w(x_i-x_j)\geq w(z_i) = w(z_j)$, 
which also appears in Theorem \ref{thm:Ribenboim}, 
cannot be deduced and could possibly be relaxed.
However, the following example shows that one cannot replace this compatibility condition by $w(x_i-x_j) \geq \min\{w(z_i),w(z_j)\}$
in Theorem \ref{thm:main}:

If $K=\mathbb{Q}(T)$ with $w$ the $T$-adic valuation,
and
$v_1$ and $v_2$ the composites of $w$
with the 2- and 3-adic valuations, respectively, 
then there is no $x$ with
$v_1(x-(2T)^{-1}) \geq v_1(T^{-1})$
and
$v_2(x) \geq 0$,
since $v_1(x)=v_1((2T)^{-1})$ implies $w(x)=-1$, but $v_2(x)\geq0$
would imply $w(x)\geq0$.
In this case
we might have sought to apply Theorem \ref{thm:main} in Situation $\mathcal{O}$ with
$S_1=\{v_1\}$, $S_2=\{v_2\}$, $x_{1}=(2T)^{-1}$, $x_{2}=0$, $z_{1}=T^{-1}$, and $z_{2}=1$.
Then assumption \condU{} is satisfied by \ref{ex:V}(1),
and $w(x_{1}-x_{2})=-1=\min\{w(z_{1}),w(z_{2})\}$.
\end{remark}

\begin{remark}
We can conclude from Theorem \ref{thm:intro_orderings_compatible}
that every field $K$ for which 
\begin{enumerate}
\item[$(\ast)$] any two orderings on $K$ induce distinct orderings on the residue field of their finest common coarsening $w$
\end{enumerate}
is an SAP-field in the sense of \cite[\S6]{Prestel},
i.e.~for every two disjoint closed subsets $S_1$ and $S_2$ of $\Sord(K)$ there is some $a\in K$ with $a>_v0$ for all $v\in S_1$ and $a<_v0$ for all $v\in S_2$.
This is in fact a special case of \cite[Theorem 9.1]{Prestel},
which in particular states that $K$ is SAP if and only if for every valuation $w$ on $K$ with formally real residue field either 
\begin{enumerate}
\item[(i)] the value group $\Gamma_w$ is $2$-divisible, or 
\item[(ii)] $|\Gamma_w/2\Gamma_w|=2$ and the residue field $Kw$ carries a unique ordering.
\end{enumerate} 
Due to the Baer-Krull theorem, our condition $(\ast)$ is precisely that (i) holds always. 
We note that Theorem \ref{thm:intro_orderings_compatible},
and therefore also Theorem \ref{thm:main}, is no longer true if we
replace $(\ast)$ by the weaker condition that $K$ is an SAP-field,
as the example $K=\mathbb{R}((T))$ with its two orderings $\leq_{0^+}$ and $\leq_{0^-}$ shows: The approximation problem with
$S_1=\{\leq_{0^+}\}$, $S_2=\{\leq_{0^-}\}$, $x_1=0$, $x_2=1$, $z_1=z_2=\frac{1}{3}$ satisfies the compatibility condition
$w(x_1-x_2)\geq w(z_1)=w(z_2)$ where $w$ is the $T$-adic valuation,
but is not solvable.
\end{remark}

\section{Applications and counterexamples}
\label{sec:value}

We start by deducing a few corollaries that
resemble similar approximation theorems in the literature.
We will phrase several of these corollaries 
for compact sets in $\mathcal{T}_{{\rm Zar}^*}$,
but recall that this property is satisfied for example
by every closed set in the constructible topology
(Remark~\ref{rem:topologies}).

\subsection{Value approximation}
The following `value approximation’ theorem
is our version of
Ribenboim’s \cite[Theorem 5]{Ribenboim},
see also 
\cite[p.~135 Th\'eor\`eme 1]{Ribenboim_book}
and
\cite[Theorem 28.12]{Warner},
which in the case of independent valuations 
appears already in Krull’s seminal paper \cite[Satz 15]{Krull}.
In \cite[Proposition 2.6.6]{ErshovMultiValuedFields} a similar result with a different condition \condU{} is given, and condition \condT{} is replaced by compactness in the Zariski topology.
\begin{corollary}
\label{thm:value_approx}
Let $S_1,\dots,S_n\subseteq\Sval(K)$, $t \in K^\times$, and $z_1,\dots,z_n\in K^\times$.
\begin{enumerate}
\item[$\condU$] Assume that Assumption \ref{assumption:f_and_t} holds for $S=S_1\cup\dots\cup S_n$ and $t$.
\item[$\condT$] Assume each $S_i$ is compact in $\mathcal{T}_{\rm Zar^*}$.
\item[$\condI$] Assume that  the $S_i$ are pairwise $t$-independent and pairwise incomparable.
\end{enumerate}
Then there exists $z \in K^\times$ with 
$$
 v(z)=v(z_i) \text{ for all }v \in S_i, \mbox{ for } i=1,\dots,n,
$$ 
if and only if
for any valuation $w$ on $K$ with a refinement in $S_i$ and a refinement in $S_j$ we have $w(z_i) = w(z_j)$.
\end{corollary}

\begin{proof}
It is clear that if such a $z$ exists then
every $w$ with a refinement in $S_i$ and a refinement in $S_j$ satisfies $w(z_i)=w(z)=w(z_j)$.
Conversely, if that compatibility condition is satisfied
then the claim follows from
Theorem \ref{thm:main} in Situation $\mathfrak{m}$
with $x_i=z_i$ for all $i$.
\end{proof}

\begin{remark}
It is worth pointing out in this context that
any form of value approximation theorem like Corollary \ref{thm:value_approx}
implies the existence of a function $\phi$ as in 
Proposition \ref{prop:phiminmax},
but of course it need in general not be given by a polynomial.
It also implies the existence of elements $b$
as in Lemma \ref{lem:pairwise_indicator_functions}.
\end{remark}

\begin{remark}\label{rem:zariski_value_approx_fails}
  The assumption ${\ref{thm:value_approx}}\condT$ cannot be replaced by compactness in the Zariski topology, as one can show with the following example in number fields -- here an obstruction to value approximation is given by the class group.
  
  Consider the number fields $K=\mathbb{Q}(\sqrt{-5})$; it is well-known that $K$ has class number $h_K=2$, with $\mathfrak{p}_{0}=(2, 1+\sqrt{-5})$ an example of a non-principal ideal, see for instance \cite[pp. 132--133]{MarcusNumberFields}.
  The extension $L=K(\sqrt{-1})$ of $K$ is of degree $2=h_K$, and one can verify that it is unramified at all places. (Note that since $\mathbb{Q}(\sqrt{-1})/\mathbb{Q}$ is unramified over all finite primes except $2$, it suffices to check that the primes of $K$ above $2$ are unramified in $L$.)
  Hence $L$ is the Hilbert class field of $K$ (see \cite[VI, Proposition 6.9]{Neu}).
  
  Consider the polynomial $f=X^2-X-1$, and let $S$ be the set of valuations on $K$ corresponding to prime ideals inert in $L/K$; no such prime ideal has residue characteristic $5$, since the prime ideal $(\sqrt{-5})$ is split in $L/K$.
  One verifies that $L/K$ is generated by a zero of $f$, so in particular $f$ is irreducible over $K$.
  The discriminant of $f$ is $5$. By standard results on the splitting of prime ideals in extensions, see \cite[I, Proposition 8.3]{Neu}, for any prime $\mathfrak p$ of $K$ inert in $L$ (which necessarily does not contain $5$), the reduction of $f$ has no zero in $\mathcal{O}_K/\mathfrak p$.
  In particular, Assumption \ref{assumption:f_and_t} is satisfied for the set $S$ and $t=1$ by Example \ref{ex:V}(1).
  
  However, Corollary \ref{thm:value_approx} does not transfer to this situation: Writing $S_1 = \{ \mathcal{O}_{\mathfrak{p}_0} \}$, $S_2 = S \setminus S_1$, we claim that there is no element $x \in K^\times$ with $v(x) = v(1+\sqrt{-5})$ for $v \in S_1$ and $v(x) = v(1)$ for $v \in S_2$.
  If $x$ were such an element, the ideal $(x) \mathfrak{p}_0^{-1}$ would be a product of prime ideals not inert in $L/K$. Since $L/K$ is unramified of degree $2$, it would be a product of prime ideals split in $L/K$, all of which are principal ideals in $K$ by the theory of the Hilbert class field (\cite[VI, Corollary 7.4]{Neu}). Hence $\mathfrak{p}_0$ itself would be principal, which is a contradiction. Therefore such an $x$ cannot exist.

  One can check that any Zariski open subset of $S$ is empty or cofinite, hence the sets $S_1$ and $S_2$ are Zariski compact.
\end{remark}

\subsection{Residue approximation}
The following `residue approximation’
appears for finitely many independent valuations 
already in \cite[Satz 17]{Krull}
(see also \cite[Lemme 6]{Ribenboim})
and for finitely many incomparable valuations
in 
\cite[p.~143 Proposition 1]{Ribenboim_book},
see also
\cite[VI.7.2 Corollary 1]{Bourbaki} and \cite[Theorem 10.2.1]{Efrat}.
A version for finitely many arbitrary valuations can be found in \cite[Lemme 11]{Ribenboim}.

\begin{corollary}
\label{thm:residue_approx}
Let $S_1,\dots,S_n\subseteq\Sval(K)$, $t \in K^\times$, and $x_1,\dots,x_n\in K$
with $x_i\in\mathcal{O}_v$ for each $v\in S_i$ and each $i$.
\begin{enumerate}
\item[$\condU$] Assume that Assumption \ref{assumption:f_and_t} holds for $S=S_1\cup\dots\cup S_n$ and $t$.
\item[$\condT$] Assume each $S_i$ is compact in $\mathcal{T}_{{\rm Zar}^*}$.
\item[$\condI$] Assume that  the $S_i$ are pairwise $t$-independent and pairwise incomparable.
\end{enumerate}
Then there exists $x \in K$ with $x\in\mathcal{O}_v$ for each $v\in S_i$ for every $i$ such that
$$
 \overline{x}=\overline{x_i} \mbox{ in } Kv \text{ for all }v \in S_i, \mbox{ for } i=1,\dots,n.
$$ 
\end{corollary}

\begin{proof}
This follows immediately from Theorem \ref{thm:main}
in Situation $\mathfrak{m}$
by choosing $z_i=1$ for all $i$.
\end{proof}

\subsection{$p$-valuations}

We now discuss approximation on the sets ${\rm S}_\pi^e(K)$,
which in particular includes the special case of
$p$-valuations of bounded $p$-ramification index, cf.~Example \ref{ex:p-vals_plus_orderings}.
We equip
${\rm S}_\pi^e(K)$ with the constructible topology.
Recall that ${\rm S}_\pi^e(K)$ is compact
and the topology coincides with 
the topology induced by $\mathcal{T}_{{\rm Zar}^*}$
(Example \ref{rem:topology.coincidence}).
Let
$$
  {\rm R}_{\pi}^{e}(K)=\bigcap_{v\in{\rm S}_{\pi}^{e}(K)}\mathcal{O}_{v}
$$
denote the corresponding holomorphy ring.

\begin{corollary}\label{cor:p-val}
Let $S_{1},\dots,S_{n}\subseteq{\rm S}_{\pi}^{e}(K)$ be disjoint and closed,
let $x_{1},...,x_{n}\in K$,
and let $z_{1},...,z_{n}\in K^{\times}$.
Assume that for any valuation $w$ on $K$ with a refinement in $S_{i}$ and a refinement in $S_{j}$ we have $w(x_{i}-x_{j})\geq w(z_{i})=w(z_{j})$.
Then there exists $x\in K$ with
$$
	v(x-x_i) > v(z_i) \text{ for all }v \in S_i, \mbox{ for } i=1,\dots,n.
$$
\end{corollary}
\begin{proof}
This follows from Theorem \ref{thm:main} in Situation $\mathfrak{m}$
using Example \ref{ex:V}(2).
Note that two distinct valuations in ${\rm S}_{\pi}^{e}(K)$ are always incomparable and $\pi$-independent (Example \ref{ex:tincomp}).
\end{proof}

\begin{corollary}
Let $S_1,\dots,S_n\subseteq{\rm S}_{\pi}^{e}(K)$ be disjoint and closed,
$x_1,\dots,x_n\in{\rm R}_{\pi}^{e}(K)$ and $k_1,\dots,k_n\in\mathbb{N}$.
Then there exists $x \in K$ with 
$$
 v(x-x_i) > v(\pi^{k_i}) \text{ for all }v \in S_i, \mbox{ for } i=1,\dots,n.
$$ 
\end{corollary}

\begin{proof}
This follows from Corollary \ref{cor:p-val}:
If $w$ is a valuation with a refinement $v_i$ in $S_i$ and a refinement $v_j$ in $S_j$, then $x_i,x_j\in{\rm R}_{\pi}^{e}(K)\subseteq \mathcal{O}_{v_i}\subseteq\mathcal{O}_w$ and $w$ is a proper coarsening of both $v_i$ and $v_j$, hence $w(x_i-x_j)\geq 0=w(\pi^{k_i})=w(\pi^{k_j})$.
\end{proof}

The argument for the following consequence is also contained in \cite[Propriété II.3.2]{DarniereThese}.

\begin{proposition}\label{prop:pvalring}
For every $\pi\in K^{\times}$ and $e>0$,
the following statements are equivalent:
\begin{enumerate}
\item For every $v\in{\rm S}_{\pi}^{e}(K)$, the holomorphy ring ${\rm R}_{\pi}^{e}(K)$ is dense in $\mathcal{O}_v$ in the $v$-topology.
\item The elements of ${\rm S}_{\pi}^{e}(K)$ are pairwise independent.
\end{enumerate}
\end{proposition}

\begin{proof}
$(1)\Rightarrow(2)$: Suppose two distinct $v_1,v_2\in{\rm S}_{\pi}^{e}(K)$
have a nontrivial common coarsening $w$.
Let $a\in K^\times$ with $w(a)>0$.
As $v_1$ and $v_2$ are incomparable, Proposition \ref{prop:RibenboimLocalities}
gives a $y\in K$ with $v_1(y)\geq v_1(\pi)>0$ and $v_2(y)\leq v_2(\pi^{-1})<0$.
If ${\rm R}_{\pi}^{e}(K)$ was dense in $\mathcal{O}_{v_1}$, then there would
exist $x\in K$ with $v_1(x)\geq0$, $v_2(x)\geq 0$ and $v_1(x-y)>v_1(a)$.
The latter condition implies that $w(x-y)\geq w(a)>0$,
but $v_2(y)<0\leq v_2(x)$ gives $v_2(x-y)<0$ and thus $w(x-y)\leq 0$,
a contradiction.

$(2)\Rightarrow(1)$: Let $v_0\in{\rm S}_{\pi}^{e}(K)$, $y\in\mathcal{O}_{v_{0}}$ and $z\in K^\times$. We want to show there exists $x\in {\rm R}_{\pi}^{e}(K)$ with $v_0(x-y)>v_0(z)$.
Without loss of generality, $v_0(z)\geq 0$.
 Let 
 $$
  S_1=\{v\in{\rm S}_{\pi}^{e}(K) : v(y)\geq 0\wedge v(z)\geq 0\}
 $$ 
 and $S_2={\rm S}_{\pi}^{e}(K)\setminus S_1$. 
Since in particular $S_1$ and $S_2$ are independent, 
we can apply Corollary~\ref{cor:p-val} to $x_1=y$, $x_2=0$, $z_1=z$, $z_2=1$ to get $x\in K$ with
$v(x-y)>v(z)\geq 0$ for $v\in S_1$, and $v(x)\geq 0$ for $v\in S_2$.
In particular $v_0(x-y)>v_0(z)$ and $x\in{\rm R}_{\pi}^{e}(K)$.
\end{proof}

\begin{example}
  An example where Proposition \ref{prop:pvalring} 
  can be applied is when $K$ is a so-called pseudo $p$-adically closed field: If we set $\pi = p$ and $e=1$, then 
  ${\rm S}_\pi^e(K)$ equals the set of $p$-valuations of $p$-rank $1$, 
  and any two of these are independent, see Theorem C and the remark following Proposition D of \cite{GeyerJardenHenselianClosuresPpC}.
\end{example}

\subsection{Comparison with strong approximation}

We now want to compare our approximation theorems with the well known results for a global field $K$. Beyond the Weak Approximation Theorem \ref{thm:artin_whaples} valid for any field, in global fields we have the following stronger result.

\begin{theorem}[{Strong Approximation, \cite[Chapter II \S15 Theorem]{Cassels}}] \label{thm:strong_approx}
  Let $K$ be a global field and $S\subsetneqq\Sabsval(K)\setminus\{v_{\rm trivial}\}$. 
  For each $v \in S$, let $x_v \in K$ and $\epsilon_v > 0$ be given such that $x_v \in \mathcal{O}_v$ and $\epsilon_v = 1$ for almost all $v \in S\cap\Sval(K)$.
  Then there exists an $x \in K$ with $\lvert x-x_v \rvert_v \leq \epsilon_v$ for all $v \in S$.
\end{theorem}
As usual in algebraic number theory, 
in this situation
a nontrivial element of $\Sabsval(K)$ is called a place.
The condition that $S$ excludes at least one place can clearly not be omitted due to the product formula (e.g.~see \cite[III, (1.3)]{Neu}) -- for instance, there is no element of $K^\times$ which is of norm $\leq 1$ at all places and of norm $<1$ at one of them.

To compare with our theorems, one first has to analyse the topologies on $S$ defined in Section~\ref{sec:balls}. One easily checks that the Zariski topology on $S$ is exactly the cofinite topology; in particular any subset of $S$ is Zariski compact.
On the other hand, the Hochster dual of the Zariski topology is the discrete topology, so $S$ is never compact unless it is finite. (It is important here that we excluded the trivial valuation from $S$.)

This means that while our approximation theorem in Situation $\mathfrak m$ is quite weak for global fields -- we can only approximate on finite sets $S_i$, i.e.\ do not obtain anything stronger than the weak approximation Theorem \ref{thm:artin_whaples} --, we can use our approximation theorem in Situation $\mathcal O$ to recover strong approximation under additional hypotheses.

\begin{proof}[Proof of Theorem \ref{thm:strong_approx} when $S$ contains no complex places and Assumption \ref{assumption:f_and_t} holds]
  For each $v\in S$ let $x_v$ and $\epsilon_v$ be given.
  Write $S_0 \subseteq S$ for the set of finite places $v \in S$ with $x_v = 1$, $\epsilon_v=1$.
  The set $S \setminus S_0$ is finite, so enumerate it as $\{ v_1, \dotsc, v_n \}$, and for each $i$ find a $z_i \in K^\times$ with $\lvert z_i \rvert_{v_i} \leq \epsilon_i$.
  Writing $z_0 = x_0 = 1$ and $S_i = \{ v_i \}$, $x_i = x_{v_i}$ for $i \geq 1$, we apply Theorem \ref{thm:main} in Situation $\mathcal{O}$ to the sets $S_i$ and elements $x_i$, $z_i$ for $i=0, \dotsc, n$.
  The element $x$ thus obtained is as desired.
\end{proof}

By Example \ref{ex:orderings_and_residue_fields}, this proves strong approximation in some situations of sets of places $S$ with density arbitrarily close to $1$, i.e.\ we only have to exclude a set of places of small density. We will see how to lift the prohibition on complex places in Section \ref{sec:exceptional_balls}.
However, we cannot reach the full statement of strong approximation, in which only a single place needs to be omitted, since for any non-constant $f \in K[X]$ and $t \in K^\times$ the Chebotarev Density Theorem 
shows that there is always a positive density of finite places $v$ of $K$ with $v(t) = 0$ in whose residue field the reduction of $f$ has a zero, hence violating Assumption \ref{assumption:f_and_t}. 

\subsection{Kronecker dimension one and reduction to finitely generated fields}

For so-called fields of Kronecker dimension one,
i.e.~algebraic extensions of $\mathbb{Q}$
and algebraic extensions of some rational function field $\mathbb{F}_q(T)$,
most of our approximation results, or some variants of it, are very easy to prove since
they can be reduced to approximation results in finite extensions of $\mathbb{Q}$ respectively $\mathbb{F}_q(T)$. 
For example, one even has the following stronger result:

\begin{proposition}
Let $K$ be a field of Kronecker dimension one, 
$S_0,\dots,S_n\subseteq\Sall(K)$ pairwise disjoint nonempty constructibly closed sets,
$x_1,\dots,x_n\in K$ and $z_1,\dots,z_n\in K^\times$.
Then there exists $x\in K$ with
$$
 x-x_i \in z_i\mathcal{O}_v\mbox{ for all }v\in S_i,\mbox{ for }i=1,\dots,n.
$$
\end{proposition}

\begin{proof}
As $\Sall(K)$ 
is the inverse limit of $\Sall(K_0)$ for the finitely generated
subfields $K_0$ of $K$, and $\Sabsval(K_0)\setminus\Sval(K_0)$ is finite 
for all of these,
the space $\Sall(K)$ is compact Hausdorff and therefore normal.
We can thus assume without loss of generality that the $S_i$
are open-closed (and still pairwise disjoint),
i.e.~of the form
$$
 S_i=\{v\in\Sall(K):a_{i1},\dots,a_{ik_i}\in\mathcal{O}_v,b_{i1},\dots,b_{il_i}\in\mathfrak{m}_v\}
$$
with elements $a_{ij},b_{ij}\in K$.
The subfield $K_0$ of $K$ generated by all $a_{ij},b_{ij},x_i,z_i$
is then without loss of generality a global field,
and
$$
 S_i':=\{v\in\Sall(K_0):a_{i1},\dots,a_{ik_i}\in\mathcal{O}_v,b_{i1},\dots,b_{il_i}\in\mathfrak{m}_v\}
$$
consists exactly of the restrictions of the elements of $S_i$ to $K_0$.
In particular, $S_0',\dots,S_n'$ are again 
nonempty and pairwise disjoint.
Now let 
$$
 T_i'=\{v\in S_{i}':x_i\in\mathcal{O}_v,z_i^{-1}\in\mathcal{O}_v\}.
$$ 
Then $S_i'\setminus T_i'$ is finite for every $i$,
so by Theorem \ref{thm:strong_approx}
there exists $x\in K_0$
such that
$x-x_i\in z_i\mathcal{O}_v$ for every $v\in S_i'\setminus T_i'$ 
and $x\in\mathcal{O}_v$ for every $v\in T_i'$,
for $i=1,\dots,n$.
This $x$ then satisfies $x-x_i\in z_i\mathcal{O}_v$ for every $v\in S_i$ and every $i=1,\dots,n$.
\end{proof}

Note that this proposition needs neither a $\condU$ condition,
since the set of localities on the global field $K_0$ 
has the property that every $a\in K_0$ lies in $\mathcal{O}_v^\times$ for almost all $v\in\Sall(K_0)$,
nor an $\condI$ condition, since localities on a global field
are automatically pairwise independent.

This might suggest that one can possibly reduce Theorem \ref{thm:main}
to the special case where the field is finitely generated over its prime field, by replacing the general field $K$ by a suitable finitely generated subfield $K_0$. 
This however does not seem to be of much use, 
mainly since in general neither of the two properties of global fields 
named in the previous paragraph holds for $K_0$;
in fact, even if a set $S\subseteq\Sall(K)$ satisfies some independence or compatibility condition,
the restriction of $S$ to $K_0$ need not.

\subsection{The $\condU$ condition}

We have seen in Remarks \ref{rem:condition_I} and \ref{rem:zariski_value_approx_fails} that in our theorems the conditions $\condI$ and $\condT$ cannot be dropped or significantly weakened. It remains to justify condition $\condU$, i.e.\ essentially Assumption \ref{assumption:f_and_t}, which is technical and may appear unnatural. 

We have seen above that the product formula alone may necessitate the omission of some place. The following example justifies our stronger assumption. We focus on a situation with only valuations, so that only conditions \ref{assumption:f_and_t}(\ref{condIntegral}, \ref{condNoApproxZero}, \ref{condLeadingTerm}) play a role.
\begin{example}\label{ex:justifying_U}\label{U_needed}
    Let $P$ be a set of prime numbers such that for every number field $L$ there exists a $p \in P$ (or equivalently infinitely many $p \in P$) such that the prime ideal $(p)$ is completely split in $L/\mathbb{Q}$. We may of course simply take $P$ to be the set of all prime numbers.

  We now consider $K = \mathbb{Q}(T)$; we will construct two Zariski compact sets $S_1, S_2$ of valuations on $K$, both only consisting of $p$-valuations (of $p$-rank $1$) for some $p \in P$, such that a certain approximation problem has no solution.
  We let $S_2$ consist of a single valuation, namely the refinement of the degree valuation on $K$ by the $q$-adic valuation on $\mathbb{Q}$ for some fixed $q \in P$.
  To construct $S_1$, first fix an enumeration $f_1, f_2, \dotsc$ of the irreducible monic polynomials in $\mathbb{Q}[T]$ and an enumeration $x_1, x_2, \dotsc$ of the non-zero elements of $K$. To each $f_i$ there is an associated discrete valuation $v_{f_i}$ on $K$, trivial on $\mathbb{Q}$.

  For each $i$, construct a valuation on $K$ in the following way. The residue field of $v_{f_i}$ is a finite extension of $\mathbb{Q}$, so by assumption on $P$ it carries a $p$-valuation $v_p$ of $p$-rank $1$ for some $p \in P$.
  We may even choose $p$ such that for all $x_j$, $1 \leq j \leq i$, with $v_{f_i}(x_j)=0$ we have $(v_p \circ v_{f_i})(x_j) = 0$,
  as this latter condition only excludes finitely many $p$.
  We let $v_i = v_p \circ v_{f_i}$, and take $S_1 = \{ v_i \colon i \geq 1 \}$.

  We claim that $S_1$, as a subspace of $\Sval(K)$ with the Zariski topology, carries the cofinite topology. To see this, observe that for every $x_j \in K^\times$, we have $v_i(x_j) = 0$ unless either $i < j$ or $v_{f_i}(x_j) \neq 0$, each of which only happens for finitely many $i$.
  Hence every nonempty Zariski-open set is cofinite, and in particular $S_1$ is compact. Since the valuations in $S_1$ are pairwise incomparable, the topology is ${\rm T}_1$ (Remark \ref{rem:T1}) and therefore the Zariski topology is exactly the cofinite topology.

  Consider now the following approximation problem: We demand an $x \in K$ such that $v(x - 0) \geq v(1)$ for all $v \in S_1$, and $w(x-T^{-1}) \geq w(T^{-2})$ for $w \in S_2$.
  Such an $x$ would be integral at all $v_{f_i}$ and furthermore integral at the degree valuation, hence necessarily constant. However, then $w(x-T^{-1}) \geq w(T^{-2})$ would be violated. Therefore this approximation problem is not solvable, in spite of $S_1$ and $S_2$ being compact sets of valuations, any two of which are independent.
\end{example}

Note that if $P$ did not satisfy our initial condition, i.e.\ if there exists a number field $L/\mathbb{Q}$ such that for no $p \in P$ is the ideal $(p)$ completely split in $L/\mathbb{Q}$, then we may as well enlarge $L$ to a totally imaginary Galois extension of $\mathbb{Q}$;
in this situation, Example \ref{ex:orderings_and_residue_fields} (where we choose $g$ to be the minimal polynomial of an integral primitive element of $L/\mathbb{Q}$) shows that Assumption \ref{assumption:f_and_t} applies with $t \in \mathbb{Q}^\times$ (or even $t=1$ since we are not interested in orderings), and therefore Theorem \ref{thm:main} is applicable to sets $S$ consisting only of valuations with residue field $\mathbb{F}_p$ for some $p \in P$ unramified in $L/\mathbb{Q}$.
Hence Example \ref{ex:justifying_U} shows that condition \ref{thm:main}\condU{} cannot be substantially weakened in Situation $\mathcal{O}$.

\subsection{Affine families of valuations}
\label{sec:Ershov}

We now briefly discuss the relation between our results
and the approximation results in the work of Ershov,
e.g.\
\cite{ErshovRRCF,ErshovMultiValuedFields,ErshovTehran}.
One of the most general results Ershov obtains is the following,
which we have paraphrased.

\begin{theorem}[{see \cite[Proposition 2.6.2.]{ErshovMultiValuedFields}}]\label{thm:Ershov.O}
  Let $S_1,\dots,S_n\subseteq\Sval(K)\setminus\{v_{\rm trivial}\}$ pairwise disjoint,
  and let $x_{1},\dots,x_{n}\in K$ and $z_{1},\dots,z_{n}\in K^{\times}$.
  Write $S = S_1 \cup\dotsb\cup S_n$.
\begin{enumerate}
    \item[$\condU$] Assume that $R = \bigcap_{v \in S} \mathcal{O}_v$ is a Prüfer ring with quotient field $K$.\footnote{In \cite[Proposition 2.6.2.]{ErshovMultiValuedFields}, the condition is that $S$ is {\em affine}, which he shows to be the case iff $S$ satisfies $\condU$, is compact in the Zariski topology and consists of pairwise incomparable valuations, see \cite[Proposition 2.3.4, Corollary 2.3.2]{ErshovMultiValuedFields}. The latter two are implied by $\condT$ and $\condI$.}
    \item[$\condT$] Assume that each $S_i$ is compact in the Zariski topology.
    \item[$\condI$] Assume that 
    the elements of $S$ are pairwise independent.
\end{enumerate}
Then there exists $x\in K$ with
$$
v(x-x_i)\geq v(z_i)\mbox{ for all }v\in S_i,\mbox{  for }i=1,\dots,n.
$$
\end{theorem}

Condition \ref{thm:Ershov.O}(U)
is satisfied for example
in the situation of 
Example \ref{ex:V}(2),
see \cite[Proposition 3]{ErshovRRCF}.
In particular, 
Theorem \ref{thm:intro_Ershov_uniformizer} follows from 
Theorem \ref{thm:Ershov.O} applied to
${\rm S}_{\pi}^{1}(K)$.
Condition \ref{thm:Ershov.O}(U) is also satisfied
in the situation of
Example \ref{ex:V}(1),
see \cite[Proposition 2.3.3]{ErshovMultiValuedFields} or \cite[Theorem 1]{RoquettePrincIdealThms}.

On the other hand, Theorem \ref{thm:Ershov.O} 
can be proven by our methods
under the stronger assumption that $R$
is not only Prüfer but satisfies the condition
explained in Remark \ref{rem:alternative.phi}.

\section{Adding finitely many exceptional localities}
\label{sec:exceptional_balls}

The standing Assumption \ref{assumption:f_and_t} on a set $S$ of localities and an element $t \in K^\times$ requires a polynomial $f \in K[X]$ such that $f(x) \not\in t\mathfrak{m}_v$ for all $x \in K$ and all $v \in S$.
If $v$ is a rank-$1$ valuation, this means in particular that $f$ has no zero in the completion of $K$ with respect to $v$, so we cannot hope to cover rank-$1$ valuations with algebraically closed completion.
For the same reason, our method as is cannot cover absolute values with completion $\mathbb{C}$, although this is desirable for analogy with Theorem \ref{thm:artin_whaples}.
It turns out, however, that at least finitely many such exceptional localities can be added to our main theorem.

The following lemma is a variant of Proposition \ref{prop:RibenboimLocalities}.
\begin{lemma}\label{lemma:finitary_weak_value_approx}
  Let $v_0, v_1, \dotsc, v_n \in \Sall(K)$ and $z_0, \dotsc, z_n \in K^\times$ such that for any valuation $w$ coarsening $v_0$ and some $v_i$, $i>0$, we have $w(z_0) \geq w(z_i)$.
  \begin{enumerate}
  \item If every $v_i$ with $i>0$ is strongly incomparable to $v_0$, then there exists $z \in K^\times$ with $z \in z_i \mathfrak{m}_{v_i}$ for all $i>0$ and $z^{-1} \in z_0^{-1} \mathfrak{m}_{v_0}$.
  \item If every $v_i$ with $i>0$ is either strongly incomparable to or a proper refinement of $v_0$, then there exists $z \in K^\times$ with $z \in z_i \mathfrak{m}_{v_i}$ for all $i>0$ and $z^{-1} \in z_0^{-1} \mathcal{O}_{v_0}$.
  \end{enumerate}
\end{lemma}

\begin{proof}
  Note that we may freely scale the $z_i$ by a common factor, so assume without loss of generality that $z_0=1$.
  
  Consider (1).
  Assume first that $v_0, \dotsc, v_n$ are dependent, say with nontrivial finest common coarsening $w$. Any $z_i$ with $w(z_i) < 0$ may be replaced by $1$, since this only strengthens the conclusion.
  Hence we have $w(z_i) = 0$ for all $i$, and we may reduce to a problem in the residue field $Kw$,
  on which $v_0,\dots,v_n$ induce localities $\bar{v}_0,\dots,\bar{v}_n$:
  Any lift $z \in K^\times$ of $\overline z \in (Kw)^\times$ satisfying $\overline z \in \overline{z}_i \mathfrak{m}_{\bar{v}_i}$ for all $i>0$ and $\overline{z}^{-1} \in \overline{z_0}^{-1} \mathfrak{m}_{\bar{v}_0}$ is as desired.
  Hence we have reduced to a problem in the residue field, where the $\bar{v_0}, \dotsc, \bar{v_n}$ are not all dependent.
  Therefore let us assume henceforth that $v_0, \dotsc, v_n$ are not all dependent.
  
  Inductively, we may first solve the problem restricted to all $v_i$ which are dependent with $v_0$, so let $z' \in K^\times$ with ${z'}^{-1} \in z_0^{-1} \mathfrak{m}_{v_0}$ and $z' \in z_i \mathfrak{m}_{v_i}$ for all $v_i$ dependent with $v_0$. Note that these conditions are then also satisfied in a $v_0$-neighbourhood of $z'$.
  Likewise, for any $v_i$ independent from $v_0$ there is a $v_i$-open set of $z'_i$ such that $z'_i \in z_j\mathfrak{m}_{v_j}$ for all $v_j$ dependent with $v_i$.
  Then Theorem \ref{thm:pairwise_indep_fin_approx} gives $z$ as desired.

  For (2), if $v_0$ is strongly incomparable to every other $v_i$, then we may solve the stronger problem (1), so assume this is not the case, i.e.\ some $v_i$ properly refines $v_0$. In particular $v_0$ is a valuation.
  Since the residue field of $v_0$ carries a nontrivial locality, it carries infinitely many pairwise independent valuations, so we may pick a refinement $v_0'$ of $v_0$ with $v_0' \vee v_i = v_0$ for any of the $v_i$ refining $v_0$. 
Note that any $w$ coarsening $v_0'$ and some $v_i$
will also coarsen $v_0=v_0'\vee v_i$, 
and hence satisfies $w(z_0)\geq w(z_i)$.  
  We can then solve problem (1) for $v_0', v_1, \dotsc, v_n$, and any solution thereof is as desired.
\end{proof}

The following results are again to be understood in the two situations $\mathcal{O}$ and $\mathfrak{m}$ as before.
This first proposition is an extension of Theorem \ref{thm:Ribenboim} to include orderings and complex absolute values.

\begin{proposition}\label{prop:finitary_approximation}
  Let $n>1$, $v_1, \dots, v_n \in \Sall(K)$, $x_1, \dotsc, x_n \in K$ and $z_1, \dotsc, z_n \in K^\times$.
  Assume that whenever $w$ coarsens $v_i$ and $v_j$, $i \neq j$, then $x_i - x_j \in z_i \mathcal{O}_w = z_j \mathcal{O}_w$.
  Furthermore assume in Situation $\mathfrak{m}$ that for any $i, j$ with $x_i \neq x_j$ the localities $v_i$ and $v_j$ are strongly incomparable; in Situation $\mathcal{O}$, assume the same only when $v_i$ and $v_j$ are orderings.
Then there exists $x \in K$ with $x-x_i \in z_i \BB_{v_i}$ for all $i$.
\end{proposition}

\begin{proof}
  In Situation $\mathcal{O}$, if $v_i$ refines $v_j$, then $x_i + z_i \BB_{v_i} \subseteq x_j + z_j \BB_{v_j}$, so we may simply remove $v_j$ from the list. We can repeat this until the $v_i$ are pairwise incomparable. If only one $v_i$ remains, the problem has become trivial, otherwise all conditions for the stronger result in Situation $\mathfrak{m}$ are satisfied, so it suffices to consider this situation.
  
  We use induction on $n$ (where $K$ varies across the class of all fields).
  If the $v_i$ are pairwise independent the statement follows from Theorem \ref{thm:pairwise_indep_fin_approx}.
  If the $v_i$ are all dependent with nontrivial finest common coarsening $w$, we may assume after scaling and shifting as in 
  Remark~\ref{rem:shifting_and_scaling} that $w(z_i)=0$ and $x_i \in \mathcal{O}_w$ for all $i$.
  It then suffices to solve the induced problem in the residue field $Kw$, i.e.\ to find $\overline{x} \in Kw$ with $\overline{x}-\overline{x_i} \in \overline{z_i}\mathfrak{m}_{\bar{v_i}}$ for all $i$, since any lift $x \in K$ of $\overline x$ will be as desired.
Solving the induced problem in $Kw$ is possible inductively, since for any $v_i$ and $v_j$ which induce the same locality on the residue field (i.e.\ are in particular not strongly incomparable), we have assumed that $x_i = x_j$, so we obtain only one condition in the residue field for $v_i$ and $v_j$.
  
  Therefore assume that the $v_i$ are neither pairwise independent nor all dependent.
  After reordering, let $v_1, \dotsc, v_k$ with $1 < k < n$ be a maximal dependent subset with nontrivial common coarsening $w$.
  By scaling and shifting we may assume that $w(z_i)=0$ and $x_i \in \mathcal{O}_w$ for all $i \leq k$.
  We can apply the induction hypothesis in the residue field of $w$ to obtain an $\overline{x} \in Kw$ which satisfies $\overline{x}-\overline{x_i} \in \overline{z_i}\mathfrak{m}_{\bar{v_i}}$ for any $i \leq k$. Take any lift $x \in K$ of $\overline{x}$.
  For $i \leq k$ we then have $x + z_i \mathfrak{m}_w \subseteq x_i + z_i \mathfrak{m}_{v_i}$.
  Hence we may replace the conditions with respect to the $v_1, \dotsc, v_k$ by a single condition with respect to $w$, reducing the number of conditions.
  Using the induction hypothesis once more proves the claim.
\end{proof}

In order to formulate our approximation theorem with finitely many exceptional localities, we introduce a modified version of Assumption \ref{assumption:f_and_t}, applying to finitely many sets $S_1, \dotsc, S_n \subseteq \Sall(K)$ and an element $t \in K^\times$.
\begin{assumption}\label{assumption:exceptional}
  There exists $f \in K[X]$ of degree $d\geq 2$, with leading coefficient $a_d$, such that conditions \ref{assumption:f_and_t}(\ref{condIntegral}, \ref{condLeadingTerm}, \ref{condPlus}) are satisfied for all $v$ in any $S_i$, and condition \ref{assumption:f_and_t}(\ref{condNoApproxZero}) is satisfied for all $v$ in any infinite $S_i$.
\end{assumption}

Note that imposing conditions \ref{assumption:f_and_t}(\ref{condIntegral}, \ref{condLeadingTerm}, \ref{condPlus}) on those $S_i$ which are finite is quite weak in practice.
For any complex absolute value $\lvert\cdot\rvert$, the conditions are satisfied for example for $f=\sum_{i=0}^da_iX^i\in K[X]$ and $t\in K^\times$
if 
 \[ \sum_{i=0}^d \lvert a_i \rvert \leq 1, \quad \lvert t \rvert + \sum_{i=0}^{d-1} \lvert a_i\rvert \leq \lvert a_d \rvert, \quad \text{and}\quad  \lvert t \rvert \leq \frac{1}{2}, \]
 which are conditions we have already seen for orderings in Example \ref{ex:V_archimedean}.
 In particular, Example~\ref{ex:orderings_and_residue_fields} extends to also cover complex absolute values in this way.
 
For a valuation $v$, conditions \ref{assumption:f_and_t}(\ref{condIntegral}, \ref{condLeadingTerm}, \ref{condPlus}) are satisfied for example if $f$ has coefficients in $\mathcal{O}_v$ and $v(t) \geq v(a_d) = 0$.

\begin{theorem}\label{thm:exceptional}
  Let $S_1, \dotsc, S_n \subseteq \Sall(K)$, $t \in K^\times$, $x_1, \dotsc, x_n \in K$ and $z_1, \dotsc, z_n \in K^\times$.
  \begin{enumerate}
  \item[\condU] Assume that Assumption \ref{assumption:exceptional} holds for $S_1, \dotsc, S_n$ and $t$.
  \item[\condT] Assume that each $S_i$ is compact in the given topology.
  \item[\condI] Assume that for any valuation $w$ on $K$ with a refinement in $S_i$ and a refinement in $S_j$ we have $w(x_i - x_j) \geq w(z_i) = w(z_j)$; assume further that the $S_i$ are pairwise $t$-independent, and in Situation $\mathfrak{m}$ furthermore pairwise incomparable.
  \end{enumerate}
  Then there exists $x \in K$ with 
  $$
   x-x_i \in z_i\BB_v\mbox{ for all }v \in S_i,\mbox{ for }i = 1, \dots, n.
  $$ 
\end{theorem}

\begin{proof}
  We may assume that all $S_i$ are nonempty. After reordering if necessary, say $S_1, \dotsc, S_k$ are infinite and $S_{k+1}, \dotsc, S_n$ are finite.
  In Situation $\mathcal{O}$, we may furthermore assume that no element of $S_i$ with $i>k$ is a coarsening of an element of $S_j$ with $j \leq k$, since such a coarsening could simply be removed.

  We first apply Proposition \ref{prop:finitary_approximation} to the elements of $S_{k+1}, \dotsc, S_n$ to obtain an element $x' \in K$ such that $x' - x_i \in tz_i \BB_v$ for all $v \in S_i$, $i > k$.
  We secondly apply Theorem \ref{thm:main} to the sets $S_1, \dotsc, S_k$, elements $x_1, \dotsc, x_k$ and $tz_1, \dotsc, tz_k$ to obtain an element $x'' \in K$ with $x''-x_i \in tz_i\BB_v$ for all $v \in S_i$, $i \leq k$.

  We want to find $x \in K$ with $x-x' \in tz_i\BB_v$ for $v \in S_i$, $i > k$, and $x-x'' \in tz_i\BB_v$ for $v \in S_i$, $i \leq k$; such an $x$ is as desired, by condition \ref{assumption:f_and_t}(\ref{condPlus}).
  Note that we have the basic compatibility condition 
  $$
   w(x'-x'') \geq \min\{w(x_i-x'), w(x_i-x_j), w(x_j-x'')\} \geq w(tz_i)=w(tz_j)
   $$ for any valuation $w$ coarsening elements of $S_i$ and $S_j$, $i > k \geq j$.
  
  If $x' = x''$, we set $x=x'$, and in the remaining case we may assume without loss of generality that $x' = 1$ and $x''=0$, by scaling and shifting all $x_i$ and $x'$ and $x''$, and scaling all $z_i$.
  It now suffices to find $b \in K^\times$ with $b \in t\mathfrak{m}_v \cap tz_i\mathfrak{m}_v$ for $v \in S_i$, $i > k$, and $b^{-1} \in t\BB_v \cap tz_i\BB_v$ for all $v \in S_i$, $i \leq k$; with such $b$ we can apply Lemma \ref{lem:approximation_lemma} to obtain $x$. (Note that by Remark \ref{rem:Lemma52_details} only condition \ref{assumption:f_and_t}(\ref{condLeadingTerm}) is necessary for $v \in S_i$, $i>k$, to apply the lemma.)
  
  To find such $b$, we imitate the proof of Lemma \ref{lem:pairwise_indicator_functions}.
  For any $i \leq k$ and $v' \in S_i$, there exists an element $b_{i,v'} \in K^\times$ with $b_{i,v'} \in t\mathfrak{m}_v \cap tz_j\mathfrak{m}_v$ for all $v \in S_j$, $j > k$, and $b_{i,v'}^{-1} \in t\BB_{v'} \cap tz_i\BB_{v'}$; this follows from Lemma \ref{lemma:finitary_weak_value_approx} applied to ${v'}$ and the $v \in S_{k+1} \cup\dotsb\cup S_n$ with an element $z_{v'} \in K^\times$ chosen to satisfy $z_{v'}\BB_{v'} = t\BB_{v'} \cap tz_i\BB_{v'}$ and elements $z_v \in K^\times$ chosen to satisfy $z_v\mathfrak{m}_v = t\mathfrak{m}_v \cap tz_j \mathfrak{m}_v$ for any $j>k$ with $v \in S_j$ (note that if $v \in S_j \cap S_{j'}$ then $z_j \mathfrak{m}_v = z_{j'} \mathfrak{m}_v$ by $\condI$).
  Note that if $w$ coarsens $v'$ and some $v\in S_j$, 
  then $w(z_i)=w(z_j)$, hence $w(z_{v'})=w(z_v)$.
  
  By compactness of the $S_i$, there exists a finite list $P$ of pairs $(i,v)$ such that for any $i \leq k$ and $v' \in S_i$ there exists $v \in S_i$ with $(i,v) \in P$ and $b_{i,v}^{-1} \in t\BB_{v'} \cap tz_i\BB_{v'}$.
  Now take $b = \phi(b_1, \dotsc, b_m)$, where 
  $b_1,\dots,b_m$
  are the $\{b_{i,v}:(i,v)\in P\}$ in arbitrary order.
  By Corollary \ref{cor:phimin}, the first part of which only requires conditions \ref{assumption:f_and_t}(\ref{condIntegral}, \ref{condLeadingTerm}), this $b$ is as desired.
\end{proof}

\section{Approximation of values of rational functions}
\label{sec:functions}

Let
$S\subseteq\Sall(K)$ be nonempty.
We write
$$
 \MM_{S}:=\bigcap_{v\in S}\MM_{v}.
$$ 
Note that
$\MM_{S}\cdot\MM_{S}\subseteq\MM_{S}$,
although $\MM_{S}$ is 
never
an element of $\B(K)$ as defined in Section~\ref{sec:balls}.
%
Let us assume throughout that $S$ does not contain the trivial valuation, is compact in $\mathcal{T}_{{\rm Zar}^*}$.
We furthermore fix $t \in K^\times$, and assume that either $S$ is finite, or $S$ and $t$ satisfy Assumption \ref{assumption:f_and_t}.
(This is for instance the case if $S_1, \dotsc, S_n$ are given which, together with $t$, satisfy Assumption \ref{assumption:exceptional}, and $S$ is a subset of some $S_i$.)
\begin{lemma}\label{lem:min_general}
  Given $z_1, \dotsc, z_n \in K^\times$ such that for any $v \in S$ we have $z_i \in \mathcal{O}_v$ for at least one $i$,
 there exists $z \in K^\times$ with $z \in \bigcap_{i=1}^n z_i \mathcal{O}_v$ for each $v \in S$.
\end{lemma}
\begin{proof}
  If $S$ is finite, say $S = \{ v_1, \dotsc, v_m \}$, then choose $z_1', \dotsc, z_m'$ to satisfy $z_j'\mathcal{O}_{v_j} = \bigcap_{i = 1}^n z_i \mathcal{O}_{v_j}$, let $v_0$ be the trivial valuation and $z_0' = 1$, and apply Lemma \ref{lemma:finitary_weak_value_approx}(2) to the $v_j$ and $z_j'$.
  If $S$ is infinite, then Assumption \ref{assumption:f_and_t} is satisfied for $S$ and $t$, and the claim follows from Corollary \ref{cor:phimin}(2) by letting $z = \phi(z_1^{-1}, \dotsc, z_n^{-1})^{-1}$.
\end{proof}

\begin{lemma}\label{lem:basis.0}
The set $\MM_S$ contains a non-zero element, and the family
$$
    \Mcal_{S}:=\big\{z\MM_{S}\;\big|\;z\in\MM_{S}\setminus\{0\}\big\}
$$
is a filter base.
\end{lemma}
\begin{proof}
First, a simple compactness argument shows that
$\MM_{S}\neq\{0\}$, as follows.
Each $v\in S$ is nontrivial,
so there exists
$z_{v}\in\mathfrak{m}_{v}\setminus\{0\}$.
By compactness in
$\mathcal{T}_{\mathrm{Zar}^{*}}$,
there are finitely many
$v_{1},\dots,v_{n}\in S$ such that
$$
    S\subseteq\bigcup_{i=1}^{n}\{v\in\Sall(K):z_{v_{i}}\in\mathfrak{m}_{v}\}.
$$
Therefore $\bigcap_{i=1}^{n}z_{v_{i}}\mathcal{O}_{v}\subseteq\mathfrak{m}_v$, for all $v\in S$.
Lemma \ref{lem:min_general} yields $z_0 \in K^\times$ with $z_{0}\in\bigcap_{i=1}^{n}z_{v_{i}}\mathcal{O}_{v}$ for each $v \in S$.
Therefore $z_{0}\in\MM_{S}\setminus\{0\}$,
and so $\Mcal_{S}$ is a nonempty family of nonempty sets.

To see that $\Mcal_{S}$ is a filter base, we let $z_{1},z_{2}\in\MM_{S}\setminus\{0\}$.
Lemma \ref{lem:min_general} again yields non-zero $z$ with $z \in z_{1}\mathcal{O}_v\cap z_{2}\mathcal{O}_v\subseteq\mathfrak{m}_v$ for each $v\in S$.
Thus $z \MM_S\in\Mcal_S$ and $z\MM_{S}\subseteq z\mathfrak{m}_v\subseteq z_1\mathfrak{m}_v\cap z_2\mathfrak{m}_v$
for every $v\in S$,
and therefore
$z\MM_S\subseteq z_{1}\MM_{S}\cap z_{2}\MM_{S}$, as required.
\end{proof}

Consequently, there is a filter $\mathcal{N}_{S,0}$ of which $\Mcal_{S}$ is a filter base.

\begin{lemma}\label{lem:PZ}
Let $\mathcal{F}$ be a filter on $K$.
Then there exists a Hausdorff field topology $\mathcal{T}$ on $K$ such that $\mathcal{F}$ is the filter of $\mathcal{T}$-neighbourhoods of $0$
if and only if the following conditions hold.
\begin{enumerate}
\item $\forall x\in K^\times\,\exists V\in\mathcal{F}\,:\,x\notin V$
\item $\forall U\in\mathcal{F}\,\exists V\in\mathcal{F}\,:\,V+V\subseteq U$
\item $\forall U\in\mathcal{F}\,\exists V\in\mathcal{F}\,:\,V\subseteq -U$
\item $\forall U\in\mathcal{F}\,\exists V\in\mathcal{F}\,:\,V\cdot V\subseteq U$
\item $\forall U\in\mathcal{F}\,\forall x\in K^{\times}\,\exists V\in\mathcal{F}\,:\,xV\subseteq U$
\item $\forall U\in\mathcal{F}\,\exists V\in\mathcal{F}\,:\,(1+V)^{-1}\subseteq 1+U$
\end{enumerate}
\end{lemma}

\begin{proof}
By \cite[Ch.~II Thm.~11.4]{Warner}, $\mathcal{F}$ satisfies (2)-(5) if and only if $\mathcal{F}$ is the filter of $\mathcal{T}$-neighborhoods of $0$ for a ring topology $\mathcal{T}$, which is then uniquely determined by $\mathcal{F}$.
In this case, by continuity of addition, the sets $1+U$ form the filter of $\mathcal{T}$-neighborhoods of $1$, and therefore (6) is equivalent to the continuity of inversion at $1$.
By continuity of multiplication, this is already equivalent to the continuity of inversion on $K^\times$.
Finally, for the filter of $\mathcal{T}$-neighborhoods of $0$, (1) is equivalent to
$\bigcap\mathcal{F}=\{0\}$,
which in any topological group holds if and only if the topology is Hausdorff (\cite[Ch.~I Thm.~1.7]{Warner}).
\end{proof}

\begin{lemma}\label{lem:basis}
The filter $\mathcal{N}_{S,0}$ satisfies conditions (1)-(6) of Lemma \ref{lem:PZ}.
\end{lemma}

\begin{proof}
It suffices to verify (1)-(6) for the filter base $\Mcal_{S}$.
Let $x\in K^{\times}$, and choose $z_0 \in \MM_S \setminus \{ 0 \}$.
Lemma \ref{lem:min_general} yields $y_0 \in K^\times$ with $y_0 \in x \mathcal{O}_v \cap z_0 \mathcal{O}_v \subseteq \mathfrak{m}_v$ for all $v \in S$.
In particular $y_{0}\in\MM_{S}\setminus\{0\}$,
and $x\notin y_0 \MM_{S}$.
This proves (1).

Turning to the other conditions,
let $z_{1}\in\MM_{S}\setminus\{0\}$.
Then $y_{1}:=z_{1}t\in\MM_{S}\setminus\{0\}$.
It follows from condition \ref{assumption:f_and_t}(iv) that
$
    \MM_{S}+\MM_{S}\subseteq t^{-1}\MM_{S}
$.
Then 
$$
    y_{1}\MM_{S}+y_{1}\MM_{S}\subseteq z_{1}\MM_{S},
$$
which verifies (2).
Condition (3) holds because $-\MM_{S}=\MM_{S}$.

Let $z_{2}\in\MM_{S}\setminus\{0\}$,
and simply choose $y_{2}=z_{2}$.
Then
\begin{align*}
    (y_{2}\MM_{S})\cdot(y_{2}\MM_{S})&=z_{2}^{2}\MM_{S}\cdot\MM_{S}\subseteq z_{2}\MM_{S},
\end{align*}
which verifies (4).

Let $z_{3}\in\MM_{S}\setminus\{0\}$ and let $x\in K^{\times}$.
Lemma \ref{lem:min_general} produces $y_3 \in K^\times$ with $y_3 \in z_3 x^{-1} \mathcal{O}_v \cap z_3 \mathcal{O}_v$ for all $v \in S$. 
Therefore $y_{3}\in\MM_{S}$, since $z_{3}\in\MM_{S}$, and
we have
$
    xy_{3}\MM_{S}\subseteq z_{3}\MM_{S}
$,
which verifies (5).

Finally,
let $z_4 \in \MM_S$, $z_4 \neq 0$.
For any valuation $v \in S$, we have $(1+z_4\mathfrak{m}_v)^{-1} \subseteq 1 + z_4\mathfrak{m}_v$.
For any other $v\in S$, we have
$(2+\mathfrak{m}_{v})^{-1}\subseteq\mathfrak{m}_{v}$.
Therefore 
$(1+\tfrac{c}{2})^{-1}=1-c(2+c)^{-1}\in1+c\mathfrak{m}_{v}$,
for all $c\in\mathfrak{m}_{v}$,
which establishes the inclusion
$(1+\tfrac{z_{4}}{2}\mathfrak{m}_{v})^{-1}\subseteq 1+z_{4}\mathfrak{m}_{v}$.
Lemma \ref{lem:min_general} gives $y_4$ with
$y_4\mathfrak{m}_{S}\subseteq z_{4}\mathfrak{m}_{S}\cap\frac{z_{4}}{2}\mathfrak{m}_{S} \subseteq \mathfrak{m}_S$.
Then
$(1+y_4\mathfrak{m}_{S})^{-1}\subseteq 1+z_{4}\mathfrak{m}_{S}$,
which verifies (6).
\end{proof}

It follows from Lemma \ref{lem:basis} that 
there is a (unique) field topology
on $K$
of which $\mathcal{N}_{S,0}$ is the filter of neighbourhoods of $0$.
We call this topology the {\em $S$-topology}.
A subset of $K$ which is open in the $S$-topology is said to be
{\em $S$-open}.
An {\em $S$-ball} is a set of the form
$$
   \mathrm{B}_{S}(x,z):=x+z\MM_{S},
$$
for $x\in K$ and $z\in K^{\times}$.

\begin{lemma}\label{lem:m_S.open}
All $S$-balls are $S$-open.
\end{lemma}
\begin{proof}
It suffices to show that $\MM_{S}$ is $S$-open. Assume first that $S$ is infinite, so $S$ contains no complex places by Assumption \ref{assumption:f_and_t} and Remark \ref{rem:ad_in_Ov}.
Let $x\in\MM_{S}$.
We must find
$z\in\MM_{S}\setminus\{0\}$
such that
$x+z\MM_{S}\subseteq\MM_{S}$.
First choose any $z_{0}\in\MM_{S}\setminus\{0\}$.
Lemma \ref{lem:min_general} affords $z \in K^\times$ with 
$z\in z_{0}(1-x)\mathcal{O}_{v}\cap z_{0}(1+x)\mathcal{O}_{v}\subseteq\mathfrak{m}_{v}$,
in particular $z\in\MM_{S}\setminus\{0\}$.

If $v\in S$ is a valuation then certainly
$x+z\mathfrak{m}_{v}\subseteq\mathfrak{m}_{v}$.
On the other hand, suppose that $v\in S$ is an ordering and let $y\in\mathfrak{m}_{v}$.
If
$0\leq_{v}x<_{v}1$, then
$1-x\leq_{v}1+x$, so
$x-1<_v z<_v 1-x$.
Otherwise if $-1<_{v}x<_{v}0$,
then
$-x-1<_vz<_v1+x$.
In either case we have
$x+zy\in\mathfrak{m}_{v}$,
and so
$x+z\mathfrak{m}_{v}\subseteq\mathfrak{m}_{v}$,
for all orderings $v\in S$.
Therefore
$x+z\MM_{S}\subseteq\MM_{S}$, as required.

It remains to treat the case of finite $S$, say $S = \{ v_1, \dotsc, v_n \}$. 
Let $x \in \MM_S$. 
Since $\mathfrak{m}_{v_i}$ is $v_i$-open, we may take $z_i \in \mathfrak{m}_{v_i}$ with $x + z_i \mathfrak{m}_{v_i} \subseteq \mathfrak{m}_{v_i}$.
Lemma \ref{lem:min_general} provides $z \in K^\times$ with $z \in z_i \mathcal{O}_{v_i}$ for all $v_i$, so in particular $z \in \MM_S$ and $x + z \MM_S \subseteq \MM_S$, as desired.
\end{proof}

It follows
from
Lemma \ref{lem:m_S.open}
that the family
$$
    \big\{{\rm B}_{S}(x,z) : x\in K,z\in\MM_{S}\setminus\{0\}\big\}
$$
is a base for the $S$-topology.
Note that the $S$-topology is finer than the $v$-topology,
for every $v\in S$.
For example, if $v\in S$ is a valuation, then
$$
	\MM_{v}=\bigcup_{x\in\MM_{v}}x+\MM_{S}.
$$
%
%
If $S=\{v\}$ is a singleton, then 
${\rm B}_{v}(x,z)={\rm B}_{\{v\}}(x,z)$,
and thus the $S$-topology coincides with the $v$-topology.

For each $m\in\mathbb{N}$, we define the {\em $S$-topology} on $K^{m}$ to be the product topology induced by the $S$-topology on $K$.
An {\em $S$-ball} in $K^{m}$ is a set of the form
$$
    \mathrm{B}_{S}(\underline{x},\underline{z}):=\prod_{i=1}^{m}\mathrm{B}_{S}(x_{i},z_{i}),
$$
for tuples $\underline{x}=(x_{1},\dots,x_{m})\in K^{m}$ and $\underline{z}=(z_{1},\dots,z_{m})\in(K^{\times})^{m}$.
We also write
$\mathrm{B}_{v}(\underline{x},\underline{z}):=\mathrm{B}_{\{v\}}(\underline{x},\underline{z})$.
For $D\subseteq K^{l}$, a function $f:D\longrightarrow K^{m}$ is said to be {\em $S$-continuous} if it is continuous with respect to the $S$-topologies.
Furthermore, $f$ is {\em $S$-hereditarily continuous} if it is $S'$-continuous for each nonempty $\mathcal{T}_{{\rm Zar}^*}$-compact $S'\subseteq S$.

\begin{example}
\label{ex:rational_function}
If $g\in K(x_1,\dots,x_l)^m$ is a tuple of rational functions given by 
$g_i=\frac{h_i}{k_i}$ with $h_i,k_i\in K[x_1,\dots,x_l]$ coprime,
then the domain
$$
 D_g = \{ \underline{x}\in K^l : k_i(\underline{x})\neq 0 \mbox{ for all } i\}
$$
of $g$ is open with respect to each $v\in S$,
and the induced map $g:D_g\longrightarrow K^m$
is $S$-hereditarily continuous, since each of the $S'$-topologies (for $S'\subseteq S$ $\mathcal{T}_{{\rm Zar}^*}$-compact) is a field topology.
\end{example}

\begin{proposition}[{}]\label{prp:intersection}
Let $S_{1},\dots,S_{n}\subseteq\Sall(K)\setminus\{v_{\mathrm{trivial}}\}$ be nonempty, and let $t\in K^\times$.
\begin{enumerate}
\item[$\condU$] 
  Assume that Assumption \ref{assumption:exceptional} holds for $S_1, \dotsc, S_n$ and $t$.
\item[$\condT$] Assume each $S_i$ is compact in $\mathcal{T}_{{\rm Zar}^*}$.
\item[$\condI$] Assume that the $S_i$ are pairwise independent.
\end{enumerate}
For each $i$, let $A_{i}\subseteq K^{m}$ be a nonempty $S_{i}$-open set.
Then
$$
    \bigcap_{i=1}^nA_{i}\neq\emptyset.
$$
\end{proposition}
\begin{proof}
We shrink each $A_{i}$ to a product of nonempty $S_{i}$-balls
$\prod_{j=1}^m\mathrm{B}_{S_i}(x_{ij},z_{ij})$,
then for each $j$ we apply Theorem~\ref{thm:exceptional}
in Situation $\mathfrak{m}$ 
to $x_{1j},\dots,x_{nj}$ and $z_{1j},\dots,z_{nj}$.
\end{proof}

\begin{theorem}
\label{thm:UA.hereditarily.continuous}
Let $S_{1},\dots,S_{n}\subseteq\Sall(K)\setminus\{v_{\mathrm{trivial}}\}$ be nonempty and pairwise disjoint,
$t\in K^\times$,
$\underline{y}_1,\dots,\underline{y}_n\in K^m$ and
$\underline{z}_1,\dots,\underline{z}_n\in (K^\times)^m$.
\begin{enumerate}
\item[$\condU$] 
  Assume that Assumption \ref{assumption:exceptional} holds for $S_1, \dotsc, S_n$ and $t$.
\item[$\condT$] Assume each $S_i$ is compact in the constructible topology.
\item[$\condI$] Assume that the elements of $S=S_1 \cup \dotsb \cup S_n$ are pairwise independent.
\end{enumerate}
Let $D\subseteq K^{l}$ be $v$-open for all $v\in S$,
and let $g:D\longrightarrow K^{m}$
be $S_i$-hereditarily continuous for all $i$.
Suppose that for each $i$ and $v\in S_i$
there exists some $\underline{x}_v\in D$ with
$$
    g(\underline{x}_{v})\in\mathrm{B}_{v}(\underline{y}_{i},\underline{z}_{i}).
$$
Then there exists $\underline{x} \in D$ with 
$$
    g(\underline{x})\in\mathrm{B}_{v}(\underline{y}_{i},\underline{z}_{i})
$$ 
for every $i$ and $v\in S_i$.
\end{theorem}

\begin{proof}
For each $i$ and each $v\in S_{i}$, the condition
$g(\underline{x}_{v})\in\mathrm{B}_{v}(\underline{y}_{i},\underline{z}_{i})$
is satisfied in an open-closed neighbourhood of $v$ in the constructible topology on $S$.
By compactness, 
we find a finite covering of $S_i$ by such open-closed sets.
By further refining this covering, we obtain a family 
$\{S_{i,1},\dots,S_{i,k_i}\}$
which is a partition of $S_{i}$ by nonempty open-closed sets, such that for each $j$ there exists $\underline{x}_{ij}\in D$ such that
$g(\underline{x}_{ij})\in\mathrm{B}_{v}(\underline{y}_{i},\underline{z}_{i})$
for each $v\in S_{i,j}$.
Therefore, for each $i,j$, we have
$g(\underline{x}_{ij})\in\mathrm{B}_{S_{i,j}}(\underline{y}_{i},\underline{z}_{i})$.
It follows from our assumptions that $D\subseteq K^l$ is $S_{i,j}$-open, and likewise it follows that $g$ is $S_{i,j}$-continuous.
Therefore the preimage $A_{i,j}:=g^{-1}(\mathrm{B}_{S_{i,j}}(\underline{y}_{i},\underline{z}_{i}))\subseteq D$ is a nonempty $S_{i,j}$-open set, for each pair $i,j$.
By Proposition \ref{prp:intersection},
there exists $x\in\bigcap_{i=1}^n\bigcap_{j=1}^{k_i}A_{i,j}$,
and this satisfies the claim.
\end{proof}

\begin{corollary}
\label{cor:rational_functions}
Let $S_{1},\dots,S_{n}\subseteq\Sall(K)\setminus\{v_{\mathrm{trivial}}\}$ be nonempty and pairwise disjoint,
$t\in K^\times$,
$\underline{y}_1,\dots,\underline{y}_n\in K^m$ and
$\underline{z}_1,\dots,\underline{z}_n\in (K^\times)^m$.
\begin{enumerate}
\item[$\condU$] 
  Assume that Assumption \ref{assumption:exceptional} holds for $S_1, \dotsc, S_n$ and $t$.
\item[$\condT$] Assume each $S_i$ is compact in the constructible topology.
\item[$\condI$] Assume that the elements of $S=S_1\cup\dots\cup S_n$ are pairwise independent.
\end{enumerate}
Let $g_1,\dots,g_m\in K(x_1,\dots,x_l)$ be rational functions.
Suppose that for each $i$ and $v\in S_i$
there exists some $\underline{x}_v\in D_g$ (cf.~Example \ref{ex:rational_function}) with
$$
    g_j(\underline{x}_{v})\in\mathrm{B}_{v}(y_{ij},z_{ij})
$$
for every $j$. Then there exists $\underline{x} \in D_g$ with 
$$
    g_j(\underline{x})\in\mathrm{B}_{v}({y}_{ij},{z}_{ij})
$$ 
for every $i,j$ and $v\in S_i$.
\end{corollary}

\begin{remark}
We remark that Corollary \ref{cor:rational_functions}
is indeed a generalization of Situation $\mathfrak{m}$
of Theorem~\ref{thm:exceptional} 
(so in particular of Theorem~\ref{thm:main})
under the stronger assumption
of pairwise independence,
as the latter one can be reobtained by taking 
the $g_j$ to be linear polynomials.
We do not know whether 
Theorem \ref{thm:UA.hereditarily.continuous}
or Corollary \ref{cor:rational_functions}
hold
without the assumption of independence under some natural compatibility condition.
\end{remark}

\begin{remark}
  Since in the last theorem and its corollary we have been working with a set $S$ of pairwise independent localities, it is possible to use  approximation results from the literature instead of our Theorem \ref{thm:exceptional},
  with a suitable adjustment of the (U) condition.
For example,
all results of this section will remain valid
for sets $S_1,\dots,S_n$ as in Theorem \ref{thm:Ershov.O}.
\end{remark}

\section*{Acknowledgements}

The authors would like to thank Marcus Tressl for several helpful discussions on spectral spaces and spaces of valuations and orderings.

Part of this work was done while all three authors were guests of the Institut Henri Poincar\'{e}
and the authors would like to thank the IHP for funding and hospitality,
and the organizers of the trimester `Model theory, combinatorics and valued fields’ for the invitation.

S.~A.~was supported by the Leverhulme Trust under grant RPG-2017-179.
P.~D.~was supported by KU Leuven IF C14/17/083.
A.~F.~was funded by the Deutsche Forschungsgemeinschaft (DFG) - 404427454.

\bibliographystyle{plain}

\end{document}